\documentclass{amsart}

\usepackage{amssymb, amsmath, color}
\usepackage{mathrsfs}
\usepackage{amscd}
\usepackage{verbatim}
\usepackage{enumerate}
\allowdisplaybreaks

\usepackage{hyperref}

\theoremstyle{plain}
\newtheorem{theorem}{Theorem}[section]
\theoremstyle{remark}
\newtheorem{remark}[theorem]{Remark}
\newtheorem{example}[theorem]{Example}
\newtheorem{facts}[theorem]{Facts}
\theoremstyle{plain}
\newtheorem{corollary}[theorem]{Corollary}
\newtheorem{lemma}[theorem]{Lemma}
\newtheorem{proposition}[theorem]{Proposition}
\newtheorem{definition}[theorem]{Definition}

\numberwithin{equation}{section}

\def\N{{\mathbb N}}
\def\Z{{\mathbb Z}}

\def\R{{\mathbb R}}
\def\C{{\mathbb C}}
\def\T{{\mathbb T}}
\def\Pol{{\mathcal P}}

\newcommand{\E}{{\mathbb E}}
\renewcommand{\P}{{\mathbb P}}
\newcommand{\F}{{\mathcal F}}
\newcommand{\A}{{\mathcal A}}

\newcommand{\M}{{\mathcal M^q(\Z)}}

\newcommand{\OO}{\mathcal{O}}

\newcommand{\eps}{\varepsilon}

\renewcommand{\O}{\Omega}

\newcommand{\D}{{\mathcal D}}
\newcommand{\calL}{{\mathscr L}}

\newcommand{\one}{{{\bf 1}}}

\newcommand{\lb}{\langle}
\newcommand{\rb}{\rangle}

\newcommand{\limn}{\lim_{n\to\infty}}

\begin{document}

\title[Stochastic parabolicity condition]{Is the stochastic parabolicity condition dependent on $p$ and $q$?}

\date\today

\author{Zdzislaw Brze\'zniak}
\address{Department of Mathematics\\University of York\\
York, YO10 5DD\\ England}
\email{zb500@york.ac.uk}

\author{Mark Veraar}
\address{Delft Institute of Applied Mathematics\\
Delft University of Technology \\ P.O. Box 5031\\ 2600 GA Delft\\The
Netherlands} \email{M.C.Veraar@tudelft.nl}




\begin{abstract}
In this paper we study well-posedness of a second order SPDE with multiplicative noise on the torus $\T = [0,2\pi]$. The equation is considered in $L^p((0,T)\times\O;L^q(\T))$ for $p,q\in (1, \infty)$. It is well-known that if the noise is of gradient type, one needs a stochastic parabolicity condition on the coefficients for well-posedness with $p=q=2$. In this paper we investigate whether the well-posedness depends on $p$ and $q$. It turns out that this condition does depend on $p$, but not on $q$. Moreover, we show that if $1<p<2$ the classical stochastic parabolicity condition can be weakened.
\end{abstract}

\thanks{The second author was supported by a VENI subsidy 639.031.930 of the Netherlands Organisation for Scientific Research (NWO)}
\maketitle

\section{Introduction}

\subsection{Setting}
Let $X$  be a separable Hilbert space with the scalar product and norm denoted respectively by $(\cdot,\cdot)$ and $\Vert \cdot \Vert$. Consider the following stochastic evolution equation on $X$:
\begin{equation}\label{eq:intro}
\left\{
  \begin{array}{rll}
    dU(t) + A U(t) \, dt &=  2 B U(t) \, dW(t) , & t\in \R_+,\\
    U(0)&=u_0.
  \end{array}
\right.
\end{equation}
Here $A$ is a linear positive self-adjoint operator with dense domain $D(A)\subseteq X$,  $B:D(A)\to D(A^{1/2})$ is a linear operator and $W(t)$, $t\geq 0$ is a real valued standard Wiener process (defined on some filtered probability space).

In \cite{KR79,Par2}, see also the monograph \cite{Rozov} and the lecture notes \cite{PreRo}, the  well-posedness of a large class of stochastic equations on $X$ has been considered, which includes equations of the form \eqref{eq:intro}.
In these papers the main assumption for the well-posedness in $L^2(\O;X)$ is:
\begin{itemize}
\item There exist $c>0$ and $K>0$ such that
\begin{equation}\label{eq:coer}
 2\|B x\| + c\|A^{1/2} x\|^2\leq (A x,x) + K\|x\|, \ \ x\in D(A).
\end{equation}
\end{itemize}
This condition will be called {\em the classical stochastic parabolicity condition}. Under condition \eqref{eq:coer} (and several others), for every  $u_0\in X$, there exists a unique solution $U\in L^2((0,T)\times\O;D(A^{1/2}))$ to \eqref{eq:intro}.
From \cite{KR79}  it is known that the condition \eqref{eq:coer} is also necessary for well-posedness, and the simple example which illustrates this, is recalled below for convenience of the reader, see \eqref{eq:kryintro}.

For Banach spaces $X$, \eqref{eq:coer} has no meaning and it has to be
reformulated. One way to do this is to assume that $A - 2 B^2$ is a ``good''
operator in $X$. There are several positive results where this assumption is
used. For instance in \cite{BCFapprx,DPIT} (in a Hilbert space setting) and
\cite{BNVW} (in a UMD Banach space setting), well-posedness for
\eqref{eq:intro} was proved. In particular, it is assumed
that $B$ is a group generator in these papers. Using It\^o's formula this allows to
reformulate \eqref{eq:intro} as a deterministic problem which can be solved
pathwise in many cases, cf. \eqref{eq:kryintro} and \eqref{eq:det}.

A widely used method to study equations of the form \eqref{eq:intro} is the Banach fixed point theorem together with the mild formulation of
\eqref{eq:intro}, see \cite{DPZ}. In order to apply this with an operator $B$ which is of half of the order of $A$ one requires maximal regularity of the stochastic convolution. To be more precise, the fixed point map $L$ of the form
\[L U (t) = \int_0^t e^{-(t-s)A} B U(s)\, dW(s)\]
has to map the adapted subspace of $L^p((0,T)\times\O;D(A))$ into itself. If one knows this, it can still be difficult to prove that $L$ is a contraction, and usually one needs that $\|B\|$ is small. Some exceptions where one can avoid this assumption are:
\begin{enumerate}[$(1)$]
\item The case where $B$ generates a group, see the previous paragraph.
\item Krylov's $L^p$-theory for second order scalar SPDEs on $\R^d$ (where $B$ is of group-type as well).
\item The Hilbert space situation with $p=2$, see \cite{KR79,Par2,Rozov} and \cite{DP82}.
\end{enumerate}
Recently, in \cite{NVW10,NVW11} a maximal regularity result for equations such as \eqref{eq:intro} has been obtained. With these results one can prove the  well-posedness results in the case $\|B\|$ is small, $X = L^q$ and $A$ has a so-called bounded $H^\infty$-calculus. A natural question is what the role of the smallness assumptions on $\|B\|$ is. In this paper we provide a complete answer to this question in the case of problem \eqref{eq:introeq} below.

\subsection{Known results for the second order stochastic parabolic equations}
In \cite{Kry}, second order equations with gradient noise have been studied.
We emphasize that the equation in \cite{Kry} is much more involved than the equation below, and we
only consider a very special case here. Consider \eqref{eq:intro} with $A =
-\Delta$ and $B = \alpha D$, where $D=\frac{\partial}{\partial x}$ and
$\alpha$ is a real constant.
\begin{equation}\label{eq:kryintro}
\left\{
  \begin{array}{rll}
    du(t) &= \Delta u(t,x) \, dt + 2 \alpha D u(t,x) \, dW(t) , & t\in \R_+, x\in \R, \\
    u(0,x)&=u_0(x), & x\in \R.
  \end{array}
\right.
\end{equation}
In this case the classical stochastic parabolicity condition \eqref{eq:coer} is $\tfrac{1}{2}(2\alpha)^2 = 2\alpha^2 <1$. Krylov proved in  \cite{Kry} and \cite{Kry00} that problem \eqref{eq:kryintro} is well-posed in $L^p(\O;L^p(\R))$ with $p\in [2, \infty)$ and in
$L^p(\O;L^q(\R))$ with $p\geq q\geq 2$, under the same assumption $2\alpha^2<1$.
In \cite[Final example]{KR79}  he  showed that if $2\alpha^2\geq 1$, then no regular solution exists. This can also be proved with the methods in \cite{BCFapprx, BNVW, DPIT}. Indeed, if $u:[0,T]\times \O\to L^q(\R)$ is a solution to \eqref{eq:kryintro}, then one can introduce a new process $v$ defined by $v(t) = e^{-B W(t)} u(t)$, $t\in \mathbb{R}_+$, where we used our  assumption that $B$ generates a group. Note that $u(t) = e^{B W(t)} v(t)$, $t\in \mathbb{R}_+$. Applying the It\^o formula one sees that $v$ satisfies the PDE:
\begin{equation}\label{eq:det}
\left\{
  \begin{array}{rll}
    dv(t) &= (1-2\alpha^2) \Delta v(t,x) \, dt, & t\in \R_+, x\in \R, \\
    v(0,x)&=u_0(x), & x\in \R.
  \end{array}
\right.
\end{equation}
Now, it is well-known from the theory of the deterministic  parabolic equations that the above problem is well-posed if and only if $2\alpha^2\leq 1$. Moreover, there is a regularizing effect if and only if $2\alpha^2<1$, see \cite[Final example]{KR79} for a different argument.

\subsection{New considerations for second order equations}
Knowing the above results it is natural to ask whether a stochastic parabolicity condition is needed for the well-posedness in $L^p(\O;L^q)$ is dependent on $p$ and $q$ or not. The aim of this paper is to give an example of an SPDE, with which one can explain the behavior of the stochastic parabolicity condition with $p$ and $q$ as parameters.
In fact we consider problem \eqref{eq:intro} with
\[A = -\Delta  \ \ \text{and} \ \ B = \alpha D + \beta |D| \ \ \ \text{on the torus $\T= [0,2\pi]$}.\]
Here $|D| = (-\Delta)^{1/2}$ and $\alpha$ and $\beta$ are real constants. This gives the following SPDE.
\begin{equation}\label{eq:introeq}
\left\{
  \begin{array}{rllr}
    du(t) &= \Delta u(t,x) \, dt &+ \  2 \alpha D u(t,x) \, d W(t) \\ & & + \ 2\beta |D| u(t,x) \, dW(t), & t\in \R_+, x\in \T, \\
    u(0,x)&=u_0(x), & &  x\in \T.
  \end{array}
\right.
\end{equation}
The classical stochastic parabolicity condition for \eqref{eq:introeq} one gets from \eqref{eq:coer} is
\begin{equation}\label{eq:classical}
\tfrac12|2\alpha i+ 2\beta|^2 =2\alpha^2+2\beta^2<1.
\end{equation}
To explain our main result let $p,q\in (1, \infty)$. In Sections \ref{sec:LpL2} and \ref{sec:LpLq} we will show that
\begin{itemize}
\item problem \eqref{eq:introeq} is well-posed in $L^p(\O;L^q(\T))$ if
\begin{equation}\label{eq:introcond}
2\alpha^2 + 2\beta^2(p-1)  < 1.
\end{equation}
\item problem  \eqref{eq:introeq} is not well-posed in $L^p(\O;L^q(\T))$ if
\[2\alpha^2+ 2\beta^2(p-1) > 1.\]
\end{itemize}
The well-posedness in $L^p(\O;L^q(\T))$  means that a solution in the sense of distributions uniquely exists and defines an adapted element of $L^p((0,T)\times\O;L^q(\T))$ for each finite $T$. The precise  concept of a solution and other definitions can be found in Sections \ref{sec:LpL2} and \ref{sec:LpLq}.

Note that $2\alpha D$ generates a group on $L^q(\T)$, whereas $2\beta |D|$
does not. This seems to be the reason the condition becomes $p$-dependent
through the parameter $\beta$, whereas this does not occur for the parameter
$\alpha$. Let us briefly explain the technical reason for the $p$-dependent
condition. For details we refer to the proofs of the main results. The
condition \eqref{eq:introcond} holds if and only if the following conditions both hold
\begin{equation}\label{eq:condparab}
2\alpha^2 - 2\beta^2  < 1,
\end{equation}
and
\begin{equation}\label{eq:condint}
\E\exp\Big(\frac{\beta^2 p |W(1)|^2}{1+2\beta^2 - 2\alpha^2}\Big)<\infty.
\end{equation}
As it will be clear from the our proofs, condition \eqref{eq:condparab} can be interpreted as a
parabolicity condition, and \eqref{eq:condint} is an integrability condition for the solution of problem \eqref{eq:introeq}.
Therefore, from now on we refer to \eqref{eq:condparab}  and \eqref{eq:condint} as the conditions for the
well-posedness in $L^p(\O;L^q)$ of problem \eqref{eq:introeq}.

Note that by taking $p\in (1,\infty)$ close to $1$, one can take $\beta^2$ arbitrary large. Surprisingly enough, such cases are not covered by the classical theory with condition \eqref{eq:classical}.

\subsection{Additional remarks}
We believe that similar results hold for equations on $\R$ instead of $\T$. However, we prefer
to present the results for $\T$, because some arguments are slightly less technical in this
case. Our methods can also be used to study higher order equations. Here
similar phenomena occur. In fact, Krylov informed the authors that with $A=
\Delta^2$ and $B = -2 \beta \Delta$, there exist
$\beta\in \R$ which satisfy $2\beta^2<1$ such that the problem
\eqref{eq:intro} is not well-posed in $L^4(\O;L^4(\R))$ (personal communication).

Our point of view is that the ill-posedness occurs, because $-2\beta \Delta$ does not generate a group on $L^4(\R)$, and therefore, integrability issues occur. With a slight variation of our methods one can check that for the latter choice of $A$ and $B$ one has the well-posedness in $L^p(\O;L^q(\R))$ for all $p\in (1, \infty)$ which satisfy $2\beta^2(p-1)<1$ and all $q\in (1, \infty)$. In particular if $\beta\in \R$ is arbitrary, one can take $p\in (1, \infty)$ small enough to obtain the well-posedness in $L^p(\O;L^q(\R))$ for all $q\in (1, \infty)$. Moreover, if $\beta$ and $p>1$ are such that $2\beta^2(p-1)>1$, then one does not have the well-posedness in $L^p(\O;L^q(\R))$. More details on this example (for the torus) are given below in Example \ref{ex:fourthorder}.

We do not present general theory in this paper, but we believe our results provides a
guideline which new theory for equations such as
\eqref{eq:intro}, might be developed.

\subsection{Organization}
This paper is organized as follows.
\begin{itemize}
\item In Section \ref{sec:pre} some preliminaries on harmonic analysis on $\T$ are given.
\item In Section \ref{sec:HS} a $p$-dependent  well-posedness result in $L^p(\O;X)$ is proved for Hilbert spaces $X$.
\item In Section \ref{sec:LpL2} we consider the well-posedness of problem \eqref{eq:introeq} in $L^p(\O;L^2(\T))$.
\item In Section \ref{sec:LpLq} the well-posedness of problem \eqref{eq:introeq} is studied in $L^p(\O;L^q(\T))$.
\end{itemize}

{\em Acknowledgment} -- The authors would like to thank Jan van Neerven and the anonymous referees and associate editors for the careful reading and helpful comments.

\section{Preliminaries\label{sec:pre}}

\subsection{Fourier multipliers}
Recall the following spaces of generalized periodic functions, see \cite[Chapter 3]{SchmTr} for details.

Let $\T=[0,2\pi]$ where we identify the endpoints. Let $\mathscr{D}(\T)$ be the space of periodic infinitely
differentiable functions $f:\T\to \C$. On $\mathscr{D}(\T)$ one can define the
seminorms $\| \cdot \|_{s}$, $s\in \N$, by
\[\|f\|_{s} = \sup_{x\in \T} |D^s f(x)|, \qquad f\in \mathscr{D}(\T),\qquad s\in \N.\]
In this way $\mathscr{D}(\T)$ becomes a locally convex space.
Its dual space $\mathscr{D}'(\T)$ is called {\em the space of periodic
distributions}. A linear functional $g:\mathscr{D}(\T)\to \C$ belongs to $\mathscr{D}'(\T)$ if and only if
there is a $N\in \N$ and a $c>0$ such that
\[|\lb f, g\rb| \leq  c \sum_{0\leq s\leq N}\|f\|_{s}.\]

For $f\in \mathscr{D}'(\T)$, we let $\hat{f}(n)= \F(f)(n) = \lb f, e_n\rb$, $n\in \Z$, where $e_n(x) = e^{-inx}$, $x\in \T$.
If $f\in L^2(\T)$ this coincides with
\[\hat{f}(n) = \F(f)(n) = \frac{1}{2\pi}\int_{\T} f(x) e^{-i x n}\, dx, \ \ n\in \Z.\]
Let $\Pol(\T)\subseteq \D(\T)$ be the space of all trigonometric polynomials.
Recall that $\Pol(\T)$ is dense in $L^p(\T)$ for all $p\in [1, \infty)$, see
\cite[Proposition 3.1.10]{GraClass}.

For a bounded sequence $m:= (m_n)_{n\in \Z}$ of complex numbers define a mapping
$T_m:\Pol(\T)\to \Pol(\T)$ by
\[T_m f(x) = \sum_{n\in \Z} m_n \hat{f}(n) e^{i n x}.\]
Let $q\in [1, \infty]$. A bounded sequence $m$ is called {\em an $L^q$-multiplier}
if $T_m$ extends to a bounded linear operator on $L^q(\T)$ if $1\leq q<\infty$ and $C(\T)$ if $q=\infty$). The space of all $L^q$-multipliers is denoted by $\M$. Moreover, we define a norm on $\M$ by
\[\|m\|_{\M} = \|T_m\|_{\calL(L^q(\T))}.\]
For more details on multipliers on $\T$ we refer to \cite{EG77} and \cite{GraClass}.

The following facts will be needed.
\begin{facts}\label{fact:mult}
\
\begin{enumerate}[(i)]
\item For all $q\in [1, \infty]$, translations are isometric in $\M$, i.e. if $\ k\in \Z$, then
\[\|n\mapsto m_{n+k}\|_{\M} = \|n\mapsto m_{n}\|_{\M}.\]
\item $\M$ is a multiplicative algebra and for all $q\in [1, \infty]$:
\[\|m^{(1)} m^{(2)}\|_{\M}\leq \|m^{(1)}\|_{\M} \|m^{(2)}\|_{\M}.\]
\item For all $q\in (1, \infty)$, $\|\one_{[0,\infty)}\|_{\M}<\infty$.
\item For all $q\in [1, \infty]$, $k\in \Z$ and $m\in \M$, $\|\one_{\{k\}}m\|_{\M}\leq \|m\|_{\M}$.
\end{enumerate}
\end{facts}

Recall the classical Marcinkiewicz multiplier theorem \cite{Marcink}, see also \cite[Theorem 8.2.1]{EG77}.
\begin{theorem}\label{thm:marcin}
Let $m=(m_n)_{n\in \Z}$ be a sequence of complex numbers and $K$ be a constant such that
\begin{enumerate}[(i)]
\item for all $n\in \Z$ one has $|m_n|\leq K$
\item for all $n\geq 1$ one has
\[\sum_{j=2^{n-1}}^{2^n-1} |m_{j+1} - m_j|\leq K, \ \ \ \text{and} \ \ \  \sum_{j=-2^{n}}^{-2^{n-1}} |m_{j+1} - m_j|\leq K.\]
\end{enumerate}
Then for every $q\in (1, \infty)$, $m\in \M$ and
\[\|m\|_{\M}\leq c_q K.\]
Here $c_q$ is a constant only depending on $q$.
\end{theorem}
In particular if $m:\R\to \C$ is a continuously differentiable function, and
\begin{equation}\label{eq:Kmarcin}
K = \max\Big\{\sup_{\xi\in \R} |m(\xi)|, \ \sup_{n\geq 1} \int_{2^{n-1}}^{2^n} |m'(\xi)| \, d\xi,  \ \sup_{n\geq 1} \int_{-2^{n}}^{-2^{n-1}} |m'(\xi)| \, d\xi \Big\}
\end{equation}
then the sequence $m=(m_n)_{n\in\Z}$, where   $m_n = m(n)$ for $n\in\Z$, satisfies the conditions of Theorem \ref{thm:marcin}.

\subsection{Function spaces and interpolation\label{sec:fspaces}}
For details on periodic Bessel potential spaces $H^{s,q}(\T)$ and Besov
spaces $B^{s}_{q,p}(\T)$ we refer to \cite[Section 3.5]{SchmTr}. We briefly
recall the definitions. For $q\in (1, \infty)$ and $s\in (-\infty, \infty)$,
let $H^{s,q}(\T)$ be the space of all $f\in \D'(\T)$ such that
\[\|f\|_{H^{s,q}(\T)} := \Big\|\sum_{k\in \Z} (1+|k|^{2})^{s/2} \hat{f}(k) e^{ikx}\Big\|_{L^q(\T)}<\infty.\]
Let $K_j = \{k\in \Z: 2^{j-1}\leq |k|<2^{j}\}$. For $p,q\in [1, \infty]$ and $s\in (-\infty,\infty)$, let
$B^{s}_{q,p}(\T)$ be the space of all $f\in \D'(\T)$ such that
\[\|f\|_{B^{s}_{q,p}(\T)} = \Big(\sum_{j\geq 0} \Big\|2^{s j} \sum_{k\in K_j} \hat{f}(k) e^{ikx}\Big\|_{L^q(\T)}^p\Big)^{1/p},\]
with the obvious modifications for $p=\infty$.
For all $q\in (1, \infty)$, $s_0\neq s_1$ and $\theta\in (0,1)$ one has the following identification of the real interpolation spaces of $H^{s,q}(\T)$, see \cite[Theorems 3.5.4 and 3.6.1.1]{SchmTr},
\begin{equation}\label{eq:interp}
(H^{s_0,q}(\T), H^{s_1,q}(\T))_{\theta,p} = B^{s}_{q,p}(\T), \ \ p\in [1,\infty), q\in (1, \infty),
\end{equation}
where $s = (1-\theta)s_0 + \theta s_1$.
Also recall that for all $q\in (1, \infty)$ one has the following continuous embeddings
\[B^{s}_{q,1}(\T) \subseteq H^{s,q}(\T)\subseteq B^{s}_{q,\infty}(\T),\]
and for all $s>r$ and $q,p\in [1, \infty]$ one has the following continuous embeddings
\[B^{s}_{q,p}\subseteq B^{s}_{q,\infty}(\T)\subseteq B^{r}_{q,1}(\T)\subseteq B^{r}_{q,p}(\T).\]

\medskip

Let $X$ be a Banach space. Assume the operator $-A$ is the a generator of an
analytic semigroup $S(t) = e^{-tA}$, $t\geq 0$, on $X$. Let us make the convention that for $\theta\in (0,1)$ and $p\in [1, \infty]$ the space $D_A(\theta,p)$ is given by all $x\in X$ for which
\begin{equation}\label{eq:realintDA}
\|x\|_{D_A(\theta,p)} := \|x\| + \Big(\int_0^1 \|t^{1-\theta} Ae^{-tA}
x\|_X^{p} \, \frac{d t}{t}\Big)^{1/p}
\end{equation}
is finite. Recall that $D_{A}(\theta,p)$ coincides with the real interpolation space
$(X,D(A))_{\theta,p}$, see \cite[Theorem 1.14.5]{Tr1}. Here one needs a modification if $p=\infty$.

Now let $X$ be a Hilbert space endowed with a scalar product $(\cdot, \cdot)$. Recall that if $A$ is a selfadjoint operator which satisfies $(A x,x) \geq 0$, then $-A$ generates a strongly continuous contractive analytic semigroup $(e^{-tA})_{t\geq 0}$, see \cite[II.3.27]{EN}. Moreover, one can define the fractional powers $A^{\frac12}$, see \cite[Section 4.1.1]{Luninterp}, and one has
\begin{equation}\label{eq:HSfract}
D_A(\tfrac12,2) = D(A^{\frac12}).
\end{equation}
This can be found in \cite[Section 1.18.10]{Tr1}, but for convenience we include a short proof. If there exists a number $w>0$ such that for all $t\geq 0$ one has $\|e^{-tA}\|\leq e^{-w t}$, then by \eqref{eq:realintDA} one obtains
\begin{align*}
\|x\|_{D_A(\theta,2)} 
 & = \|x\| + \Big(\int_0^1 (A e^{-2tA} y, y) \, dt\Big)^{1/2} = \|x\| + \Big(\|y\|^2 - \|e^{-A} y\|^2\Big)^{1/2}
\end{align*}
where $y = A^{1/2} x$. Since $\|e^{-A} y\|\leq \|e^{-w} y\|$, one has
\[ C_w  \|A^{1/2}x\| \leq \Big(\|y\|^2 - \|e^{-A} y\|^2\Big)^{1/2} \leq \|A^{1/2} x\|.\]
We conclude $D_A(\frac12,2) = D(A^{\frac12})$ with the additional assumption on the growth of $\|e^{tA}\|$. The general case follows from $D_{A}(\frac12,2) = D_{A+1}(\frac12,2) = D((A+1)^{\frac12}) = D(A^{\frac12})$, see \cite[Lemma 4.1.11]{Luninterp}.

Finally we recall that for a Banach space $X$ and a measure space $(S,\Sigma,\mu)$, $L^0(S; X)$ denotes the vector space of strongly measurable functions $f:S\to X$. Here we identify functions which are equal almost everywhere.

\section{Well-posedness in Hilbert spaces\label{sec:HS}}

\subsection{Solution concepts}
Let $(\Omega, \A, \P)$ be a probability space with a filtration $\mathbb{F}=(\F_t)_{t\geq
0}$. Let $W:\R_+\times\O\to \R$ be a standard $\mathbb{R}$-valued $\mathbb{F}$-Brownian
motion. Let $X$ be a separable Hilbert space. Consider the following abstract
stochastic evolution equation:
\begin{equation}\label{eq:SPDEFA}
\left\{
  \begin{array}{rll}
    dU(t) + A U(t) \, dt &=  2 B U(t) \, dW(t) , & t\in \R_+, \\
    U(0)&=u_0.
  \end{array}
\right.
\end{equation}
Here we assume the operator $-A$ is the a generator of an
analytic strongly continuous semigroup $S(t) = e^{-tA}$ on $X$, see \cite{EN} for details, $B:D(A)\to D(A^{1/2})$ is bounded and linear and $u_0:\O\to X$ is $\F_0$-measurable.

The following definitions are standard, see e.g. \cite{DPZ} or \cite{NVW11}.

\begin{definition}\label{def:strong}
Let $T\in (0,\infty)$. A process $U:[0,T]\times\O\to X$ is called {\em a strong solution of \eqref{eq:SPDEFA} on $[0,T]$} if and only if
\begin{enumerate}[(i)]
\item $U$ is strongly measurable and adapted.
\item one has that $U\in L^0(\O; L^1(0,T;D(A)))$ and $B(U) \in L^0(\O;L^2(0,T;X))$,
\item $\mathbb{P}$-almost surely, the following identity holds in $X$:
\[U(t) - u_0 = \int_0^t A U(s) \, ds + \int_0^t 2B U(s) \, dW(s), \ \ \ \ t\in [0,T].\]
\end{enumerate}
Let $t_0\in (0,\infty]$. A process $U:[0,t_0)\times\O\to X$ is called {\em a strong solution of \eqref{eq:SPDEFA} on $[0,t_0)$} if
for all $0<T<t_0$ it is a strong solution of \eqref{eq:SPDEFA} on $[0,T]$.
\end{definition}
From the definition it follows that if a process $U:[0,t_0)\times\O\to X$ is  a strong solution of \eqref{eq:SPDEFA} on $[0,t_0)$, then
\[U\in L^0(\O;C([0,T];X)),  \ \ T<t_0.\]

\begin{definition}\label{def:mild}
Let $T\in (0,\infty)$. A process $U:[0,T]\times\O\to X$ is called a {\em mild solution of \eqref{eq:SPDEFA} on $[0,T]$} if and only if
\begin{enumerate}[(i)]
\item $U$ is strongly measurable and adapted,
\item one has $B U\in L^0(\O; L^2(0,T;X))$,
\item for all $t\in [0,T]$, the following identity holds in $X$:
\[U(t) = e^{tA} u_0 + \int_0^t e^{(t-s)A}  2B U(s) \, dW(s), \ \ \ \text{almost surely}.\]
\end{enumerate}
Let $t_0\in (0,\infty]$. A process $U:[0,t_0)\times\O\to X$ is called {\em a mild solution of \eqref{eq:SPDEFA} on $[0,t_0)$} if
for all $0<T<t_0$ it is a mild solution of \eqref{eq:SPDEFA} on $[0,T]$.
\end{definition}

The following result is well-known, see \cite{DPZ}.
\begin{proposition}\label{prop:equiv}
Let $T\in (0,\infty)$. Assume $u_0 \in L^0(\O,\F_0;X)$. For a process $U:[0,T]\times \O\to X$ the following statements are equivalent:
\begin{enumerate}[$(1)$]
\item $U$ is a strong solution of \eqref{eq:SPDEFA} on $[0,T]$.
\item $U$ is a mild solution of \eqref{eq:SPDEFA} on $[0,T]$ and
\[U\in L^0(\O;C([0,T];X))\cap L^0(\O;L^1(0,T;D(A))).\]
\end{enumerate}
\end{proposition}

\begin{definition}\label{def:Lp}
Let $p\in (1, \infty)$.
\begin{enumerate}[$(1)$]
\item Let $T\in (0,\infty)$. A process $U:[0,T]\times\O\to X$ is called {\em an $L^p(X)$-solution of
\eqref{eq:SPDEFA} on $[0,T]$} if it is a strong solution on $[0,T]$ and
$U\in L^p((0,T)\times\O;D(A))$.

\item Let $t_0\in (0,\infty)$. A process $U:[0,t_0)\times\O\to X$ is called {\em an $L^p(X)$-solution of
\eqref{eq:SPDEFA} on $[0,t_0)$} if for all $0<T<t_0$ it is an {\em an $L^p(X)$-solution of
\eqref{eq:SPDEFA} on $[0,T]$}.
\end{enumerate}
\end{definition}

To finish this section we give a definition of the well-posedness for \eqref{eq:SPDEFA}.

\begin{definition}
Let $p\in [0, \infty)$.
\begin{enumerate}[$(1)$]
\item Let $T\in (0,\infty)$. The problem \eqref{eq:SPDEFA} is called {\em well-posed in $L^p(\O;X)$ on $[0,T]$} if for each $u_0\in L^p(\O;D(A))$ which is $\F_0$-measurable, there exists a unique $L^p(X)$-solution of \eqref{eq:SPDEFA} on $[0,T]$.

\item Let $t_0\in (0,\infty]$. The problem \eqref{eq:SPDEFA} is called {\em well-posed in $L^p(\O;X)$ on $[0,t_0)$} if for each $u_0\in L^p(\O;D(A))$ which is $\F_0$-measurable and there exists a unique $L^p(X)$-solution of \eqref{eq:SPDEFA} on $[0,t_0)$. If $t_0=\infty$, we will also call the latter {\em well-posed in $L^p(\O;X)$}.
\end{enumerate}
\end{definition}

\subsection{Well-posedness results}
For the problem \eqref{eq:SPDEFA} we assume the following.
\begin{enumerate}
\item[\textbf{(S)}]\label{cond:S} The operator $C:D(C)\subset X\to X$ is skew-adjoint, i.e. $C^* = -C$, and that
\[A = C^*C, \ \ \text{and}  \ \ B =  \alpha C+\beta |C|, \ \ \text{for some $\alpha,\beta\in \R$}.\]
To avoid trivialities assume that $C$ is not the zero operator.
\end{enumerate}
Using the spectral theorem, see \cite[Theorem VIII.4, p. 260]{ReSi},  one can see that $|C| = A^{1/2}$ and $D(B) = D(|C|) = D(C)$.

Under the assumption \textbf{(S)}, the operator $-A$ is the generator of an analytic contraction
semigroup $S(t) = e^{-tA}$, $t\geq 0$, on $X$.
Moreover, $(e^{tC})_{t\in \R}$ is a unitary group. In this situation we can prove
the first $p$-dependent the well-posedness result.

\begin{theorem}\label{thm:wellposed}
Assume the above condition {\rm \textbf{(S)}}. Let $p\in[2,\infty)$.
If $\alpha,\beta\in\R$ from \eqref{eq:SPDEFA} satisfy
\begin{equation}\label{eq:condsigma}
2\alpha^2+2\beta^2(p-1) <1,
\end{equation}
then for every $u_0\in L^p(\O,\F_0;D_A(1-\frac1p,p))$, there exists a unique $L^p(X)$-solution $U$ of \eqref{eq:SPDEFA} on $[0,\infty)$. Moreover, for every $T<\infty$ there is a constant $C_T$ independent of $u_0$ such that
\begin{align}
\label{eq:Lptime}\|U\|_{L^p((0,T)\times\O;D(A))} &\leq C_T \|u_0\|_{L^p(\O;D_A(1-\frac1p,p))},
\\ \label{eq:Ctime} \|U\|_{L^p(\O;C([0,T];D_A(1-\frac1p,p)))} &\leq C_T \|u_0\|_{L^p(\O;D_A(1-\frac1p,p))}.
\end{align}
\end{theorem}

\begin{remark}
The classical parabolicity condition for \eqref{eq:SPDEFA} is $\tfrac{1}{2}((2\alpha)^2+(2\beta)^2) = 2\alpha^2 + 2\beta^2 <1$.
This condition is recovered if one takes $p=2$ in \eqref{eq:condsigma}.
Recall from \eqref{eq:HSfract} that $D_A(\frac12,2) = D(A^{\frac12})$ for $p=2$.
Surprisingly, Theorem \ref{thm:wellposed} is optimal in the sense that for every $p\geq 2$ the
condition \eqref{eq:condsigma} cannot be improved in general. This will be proved
in Theorem \ref{thm:sharpex}. Note that if $\beta=0$, then the condition
\eqref{eq:condsigma} does not depend on $p$. This explains why in many papers
the $p$-dependence in the well-posedness of SPDEs in $L^p(\O;X)$ is not visible, see \cite{BNVW,
DPIT,Kry,Kry00}. Note that if $\beta=0$, then $B$ generates a group. This is the main structural assumption which seems to be
needed to obtain a $p$-independent theory.
\end{remark}

\begin{proof}[Proof of Theorem \ref{thm:wellposed}]
If necessary, we consider the complexification of $X$ below.
By the spectral theorem (applied to $-iC$), see \cite[Theorem VIII.4, p. 260]{ReSi},  there exists a $\sigma$-finite measure space $(\OO, \Sigma,\mu)$, a
measurable function $c:\OO\to \R$ and a unitary operator $Q:X\to L^2(\OO)$ such that $Q C Q^{-1} = i c$.
Define the measurable functions $a:\OO\to [0,\infty)$ and $b:\OO\to \C$ by $a = c^2$ and $b = \beta |c| + i \alpha c$.
In this case one has $Q e^{tC} Q^{-1} = e^{i t c}$,  $Q S(t)Q^{-1} = e^{- t a}$ and $Q B Q^{-1} = b$.
The domains of the multiplication operators are as usual, see \cite{EN}.

Formally, applying $Q$ on both sides of \eqref{eq:SPDEFA} and denoting $V = Q U$ yields the following family of stochastic equations for $V$:
\begin{equation}\label{eq:SPDEFourier}
\left\{
  \begin{array}{rll}
    d V(t) + a V (t) \, dt &= 2 b V(t) \, dW(t), & t\in \R_+, \\
    V(0)&=v_0,
  \end{array}
\right.
\end{equation}
where $v_0 = Q u_0$. It is well-known from the theory of SDE that for fixed
$\xi\in \OO$, \eqref{eq:SPDEFourier} has a unique solution $v_{\xi}:\R_+\times
\O\to \R$ given by
\[v_{\xi}(t) = e^{- t a(\xi) - 2t b^2(\xi)} e^{2b(\xi) W(t)} v_0(\xi).\]
Indeed, this follows from the (complex version of) It\^o's formula, see
\cite[Chapter 17]{Kal}. Clearly, $(t,\omega,\xi)\mapsto v_{\xi}(t,\omega)$ defines a jointly measurable mapping.
Let $V:\R_+\times\O\to L^0(\OO)$ be defined by
$V(t,\omega)(\xi) = v_{\xi}(t,\omega)$. We check below that actually
$V:\R_+\times\O\to L^2(\OO)$ and
\begin{align}
\label{eq:LptimeV}\|V\|_{L^p((0,T)\times\O;D(a))} &\leq C_T \|u_0\|_{L^p(\O;D_A(1-\frac1p,p))}.
\end{align}
Let us assume for the time being \eqref{eq:LptimeV} has been proved.
Then the adaptedness of process $V:[0,T]\times\O\to L^2(\mathcal{O})$ follows from its definition.
In particular, $a V, b V\in L^p((0,T)\times\O;L^2(\mathcal{O}))$ and since $p\geq 2$ we get $aV\in L^1(0,T;L^2(\mathcal{O}))$ a.s.\ and $b V\in L^2((0,T)\times\O;L^2(\mathcal{O}))$. Using the facts that for all $t\in [0,T]$ and $\P$-almost surely
\[\int_0^t a(\xi) v_{\xi}(s) \, ds = \Big(\int_0^t a V(s) \, ds\Big)(\xi), \ \ \ \text{for almost all $\xi\in \mathcal{O}$},\]
\[\int_0^t b(\xi) v_{\xi}(s) \, d W(s) = \Big(\int_0^t b V(s) \, d W(s)\Big)(\xi), \ \ \ \text{for almost all $\xi\in \mathcal{O}$},\]
one sees that $V$ is an $L^p(L^2(\mathcal{O}))$-solution of \eqref{eq:SPDEFourier}. These facts can be rigorously justified  by a standard approximation argument. Using the above facts one also sees that uniqueness of $V$ follows from the uniqueness of $v_{\xi}$ for each $\xi\in \mathcal{O}$.
Moreover, it follows that the process $U = Q^{-1}V$ is an $L^p(X)$-solution
of \eqref{eq:SPDEFA} and inequality \eqref{eq:Lptime} follows from inequality
\eqref{eq:LptimeV}. Moreover, $U$ is the
unique $L^p(X)$-solution of \eqref{eq:SPDEFA}, because any other $L^p(X)$-solution
$\tilde{U}$ of \eqref{eq:SPDEFA} would give an $L^p(L^2(\mathcal{O}))$-solution $\tilde{V} = Q
\tilde{U}$ of \eqref{eq:SPDEFourier} and by uniqueness of the solution of \eqref{eq:SPDEFourier} this yields $V = \tilde{V}$ and therefore, $U = \tilde{U}$.

Hence to finish the proof of the Theorem we have to prove inequalities \eqref{eq:LptimeV} and \eqref{eq:Ctime}.

{\em Step 1} - Proof of \eqref{eq:LptimeV}.

Fix $t\in \R_+$ and $\omega\in \O$. Then using $|e^{i x}| = 1$, one gets
\begin{align*}
\|(1+a)V(t,\omega)\|_{L^2(\OO)}^2 &=  \int_{\OO} (1+a)^2 |V(t,\omega)|^2 \, d\mu
\\ & = \int_{\OO} (1+a)^2 e^{- 2 (1+2\beta^2 -2 \alpha^2) ta } e^{4\beta |c| W(t,\omega)} |v_0|^2 \, d\mu.
\end{align*}
Put $\theta := \beta^2 - \alpha^2$. Let $\eps\in (0,1)$ be such that
$2\beta^2(p-1)+2\alpha^2<1-\eps$. Put $r = 1-\eps$. Then $r+ 2\theta> 2\beta^2 p$ and  one can
write
\begin{align*}
\|(1+a)V(t,\omega)\|_{L^2(\OO)}^2 & = \int_{\OO} (1+a)^2 e^{- 2 (r+2\theta)t a} e^{4\beta |c| W(t,\omega)} e^{-2\eps t a} |v_0|^2\, d\mu.
\end{align*}
Now using $c^2 =a$ one gets
\begin{align*}
- 2 (r+2\theta)t a + 4\beta |c| W(t,\omega) & = -2 (r+2 \theta)t  \Big[|c| - \frac{\beta W(t,\omega)}{(r+2\theta)t}\Big]^2 + \frac{2\beta^2 |W(t,\omega)|^2}{(r+2\theta)t}
\\ & = - f(t) [|c|-g(t,\omega)]^2 + 2h(t,\omega),
\end{align*}
where
\[f(t) = 2 (r+2\theta)t,  \ \ g(t,\omega) = \frac{\beta W(t,\omega)}{(r+2\theta)t},  \ \ h(t,\omega) = \frac{\beta^2 |W(t,\omega)|^2}{(r+2\theta)t}.\]
It follows that
\begin{equation}\label{eq:isometryU}
\begin{aligned}
\|(1+a)V(t,\omega)\|_{L^2(\OO)}^2 =  \int_{\OO}  e^{-f(t) (|c| - g(t,\omega))^2} e^{2 h(t,\omega)}   e^{-2\eps t a} (1+a)^2 |v_0|^2\, d\mu.
\end{aligned}
\end{equation}
Since $e^{-f(t) (|c| - g(t,\omega))^2}\leq 1$, this implies that
\begin{align*}
\|(1+a)V(t,\omega)\|_{L^2(\OO)}^2  \leq  e^{2 h(t,\omega)} \int_{\OO} (1+a)^2 e^{-2\eps t a} |v_0|^2\, d\mu.
\end{align*}
Using the independence of $v_0$ and $(W(t))_{t\geq 0}$ it follows that that
\begin{equation}\label{eq:pointwiseestV}
\begin{aligned}\E \|(1+a)V(t,\omega)\|_{L^2(\OO)}^p  & \leq  \E \Big(e^{p h(t,\omega)} \Big(\int_{\OO} (1+a)^2 e^{-2\eps t a} |v_0|^2\, d\mu\Big)^{p/2}
\\ &  = \E e^{p h(1,\omega)} \|(1+a) e^{-\eps t a} v_0\|_{L^p(\O;L^2(\OO))}^p,
\end{aligned}
\end{equation}
where we used $\E e^{p h(t)} = \E e^{p h(1)}$.
Integrating over the interval $[0,T]$, it follows from \eqref{eq:pointwiseestV} and \eqref{eq:realintDA} that there exists a constant $C$ is independent of $u_0$ such that
\begin{align*}
\Big(\int_0^T&  \E \|(1+a)V(t)\|_{L^2(\OO)}^p \, dt\Big)^{1/p} \\ & \leq \big(\E e^{p h(1,\omega)}\big)^{1/p} \Big(\int_0^T \|(1+a) e^{-\eps t a} v_0\|_{L^p(\O;L^2(\OO))}^p \, dt\Big)^{1/p}
\\ & =\big(\E [e^{p h(1)}]\big)^{1/p} \Big(\E \int_0^T \|(1+A) e^{-\eps t A} u_0\|_X^{p} \, dt\Big)^{1/p}
\\ & \leq C \big(\E [e^{p h(1)}]\big)^{1/p}  \|u_0\|_{L^p(\O;D_A(1-\frac1p,p))}.
\end{align*}
 One has $\E e^{p h(1)}<\infty$ if and only if $\frac{p
\beta^2}{(r+2\theta)}<\frac12$. The last inequality is satisfied by assumptions since it is equivalent to
$2\beta^2(p-1) + 2\alpha^2 <r=1-\eps$. It
follows that $V\in L^p((0,T)\times\O;D(a))$ for any $T\in (0,\infty)$, and
hence \eqref{eq:LptimeV} holds. From this we can conclude that $V$ is an
$L^p(L^2(\mathcal{O}))$-solution on $\R_+$.

{\em Step 2} - Proof of \eqref{eq:Ctime}. By Step 1 and the preparatory observation the process $U$ is a strong $L^p(X)$ solution of \eqref{eq:SPDEFA}.
By Proposition \ref{prop:equiv}, $U$ is a mild solution of \eqref{eq:SPDEFA} as well and hence
\[U(t) = e^{t A} v_0 + \int_0^t e^{(t-s)A} 2 B U(s) \, dW(s), \ \ \ t\in [0,T].\]
Since $u_0\in L^p(\O;D_{A}(1-\frac1p,p))$, it follows from the strong continuity of $e^{tA}$ on $D_{A}(1-\frac1p,p)$, see \cite[Proposition 2.2.8]{Lun}, that
\[\|t\mapsto e^{tA} u_0\|_{L^p(\O;C([0,T];D_{A}(1-\frac1p,p)))} \leq C \|u_0\|_{L^p(\O;D_{A}(1-\frac1p,p))}.\]
Since by \eqref{eq:LptimeV}, $B U\in L^p((0,T)\times\O;D(a^{1/2}))$, it follows with \cite[Theorem 1.2]{NVW10} that
\begin{align*}
\Big\|t\mapsto \int_0^t e^{(t-s)A} & 2B U(s) \, dW(s)\Big\|_{L^p(\O;C([0,T];D_{A}(1-\frac1p,p)))} \\ & \leq C_1 \|BV\|_{L^p((0,T)\times\O;D(A^{1/2}))} \leq C_2 \|u_0\|_{L^p(\O;D_{A}(1-\frac1p,p))}.
\end{align*}
Hence \eqref{eq:Ctime} holds, and this completes the proof. Note that the assumptions in \cite[Theorem 1.2]{NVW10} are satisfied since $A$ is positive and self-adjoint.
\end{proof}

\begin{remark}
If one considers $A = \Delta$ on $L^2(\T)$ or $L^2(\R)$, then for the unitary operator $Q$ in the above proof one can take the discrete or continuous Fourier transform.
\end{remark}

The above proof one has a surprising consequence. Namely, the proof of \eqref{eq:LptimeV} also holds if the number $p$ satisfies $1<p<2$. With some additional argument we can show that in this situation there exists a unique $L^p(X)$-solution $U$ of \eqref{eq:SPDEFA}. This also implies that we need less than the classical stochastic parabolicity condition one would get from \eqref{eq:coer}. Indeed, \eqref{eq:coer} gives $2\alpha^2 + 2\beta^2 <1$. For the well-posedness in $L^p(\O;X)$, we only require \eqref{eq:condsigma} which, if $1<p<2$,      is less restrictive than $2\alpha^2 + 2\beta^2<1$.
In particular, note that if $2\alpha^2<1$, and $\beta\in \R$ is arbitrary, then \eqref{eq:condsigma} holds if we take $p$ small enough.

\begin{theorem}\label{thm:wellposed2}
Let $p\in(1,\infty)$.
If the numbers $\alpha,\beta\in\R$ from \eqref{eq:SPDEFA} satisfy \eqref{eq:condsigma},
then for every $u_0\in L^p(\O,\F_0;D_A(1-\frac1p,p))$, there exists a unique $L^p(X)$-solution $U$ of \eqref{eq:SPDEFA} on $[0,\infty)$. Moreover, for every $T<\infty$ there is a constant $C_T$ independent of $u_0$ such that
\begin{align}
\label{eq:Lptimecor}\|U\|_{L^p((0,T)\times\O;D(A))} &\leq C_T \|u_0\|_{L^p(\O;D_A(1-\frac1p,p))}
\end{align}
\end{theorem}
We do not know whether \eqref{eq:Ctime} holds for $p\in (1, 2)$.
However, since $U$ is a strong solution one still has that $U\in L^p(\O;C([0,T];X))$.
\begin{proof}
The previous proof of \eqref{eq:LptimeV} still holds for $p\in (1,2)$, and hence if we again define $U = Q^{-1}V$, the estimate \eqref{eq:Lptimecor} holds as well.
To show that $U$ is an $L^p(X)$-solution, we need to check that it is a
strong solution. For this it suffices to show that $B U\in L^p(\O;
L^2(0,T;X))$. Since $\|b V\|_{L^2(\OO)} = \|B U\|_X$, it is equivalent to show
that $b V\in L^p(\O; L^2(0,T;L^2(\OO)))$, where we used the notation of the
proof of Theorem \ref{thm:wellposed}. Now after this has been shown, as in the
proof of Theorem \ref{thm:wellposed} one gets that $U$ is a strong solution of
\eqref{eq:SPDEFA}.

By \eqref{eq:pointwiseestV}, for all $t\in (0,T]$ one has $V(t)\in L^p(\O;D(a))$ and
\begin{equation}\label{eq:pointwiseagain}
\|(1+a)V(t)\|_{L^p(\O;L^2(\OO))}   \leq  C_t \|v_0\|_{L^p(\O;L^2(\OO))}.
\end{equation}
Applying \eqref{eq:SPDEFourier} for each $t\in (0,T]$ and $\xi\in \OO$ yields that
\begin{equation}\label{eq:solbeq}
\int_0^t 2b(\xi) v_{\xi}(s)\, d W(s) = v_{\xi}(t) - v_{\xi}(0)+ \int_0^t a(\xi) v_{\xi}(s)\, d s := \eta_t(\xi).
\end{equation}
We claim that $b V\in L^p(\O;L^2(0,T;L^2(\OO)))$, and for all $t\in [0,T]$,
\[\int_0^t 2b V(s)\, d W(s) = \eta_t,\]
where the stochastic integral is defined as an $L^2(\OO)$-valued random variable, see Appendix \ref{sec:stochintLq}.

To prove the claim note that $\eta_t\in L^p(\O;L^2(\OO))$ for each $t\in (0,T]$. Indeed, by \eqref{eq:pointwiseagain} and \eqref{eq:LptimeV}
\begin{align*}
\|\eta_t\|_{L^p(\O;L^2(\OO))} & \leq \|V(t)\|_{L^p(\O;L^2(\OO))}+ \|V(0)\|_{L^p(\O;L^2(\OO))} \\ & \qquad  \qquad  \qquad + \int_0^t \|a V(s)\|_{L^p(\O;L^2(\OO))} \, d s
\\ & \leq (C_t +1)\|v_0\|_{L^p(\O;L^2(\OO))} + t^{1-\frac{1}{p}} \|a
V\|_{L^p((0,T)\times\O;L^2(\OO))}
\\ & \leq C_{t,T} \|v_0\|_{L^p(\O;D_a(1-\frac1p,p))}<\infty.
\end{align*}
Therefore, by \eqref{eq:solbeq} and Lemma \ref{lem:suffstochint} (with $\phi = 2b V$ and $\psi = 2b v$), the claim
follows, and from \eqref{eq:itoisom} we obtain
\begin{align*}
\|2b V\|_{L^p(\O;L^2(0,T;L^2(\OO)))} &\leq c_{p,2} \|\eta_T\|_{L^p(\O;L^2(\OO))}
\leq c_{p,2}C_{T,T} \|v_0\|_{L^p(\O;D_a(1-\frac1p,p))}
\end{align*}
\end{proof}

An application of Theorems \ref{thm:wellposed} and \ref{thm:wellposed2} is given in Section \ref{sec:LpL2}, where it is also be shown that the condition \eqref{eq:condsigma} is sharp.

Next we present an application to a fourth order problem.
\begin{example}\label{ex:fourthorder}
Let $s\in \R$. Let $\beta\in \R$. Consider the following SPDE on $\T$.
\begin{equation}\label{eq:kryintro2}
\left\{
  \begin{array}{rll}
    du(t,x) + \Delta^2 u(t,x) \, dt &= -2 \beta \Delta u(t,x) \, dW(t), & t\in \R_+, x\in \T,
    \\ D^k u(t,0) &= D^k u(t,2\pi), &  t\in \R_+, k\in \{0,1, 2, 3\}
\\    u(0,x)&=u_0(x), & x\in \T.
  \end{array}
\right.
\end{equation}
Let $U:\R_+\times\O\to H^{s,2}(\T)$ be the function given by $U(t)(x) =
u(t,x)$. Then \eqref{eq:kryintro2} can be formulated as \eqref{eq:SPDEFA} with
$C = i \Delta$ and $X=H^{s,2}(\T)$. If we take $p\in (1, \infty)$, such that
$2\beta^2(p-1)<1$, then for all $u_0\in L^p(\O,\F_0;B^{s+4-\frac4p}_{2,p}(\T))$,
\eqref{eq:kryintro2} has an $L^p$-solution, and
\[\|U\|_{L^p((0,T)\times\O;H^{s+4,2}(\T))} \leq C_T \|u_0\|_{L^p(\O;B^{s+4-\frac4p}_{2,p}(\T))},\]
where $C_T$ is a constant independent of $u_0$.

It should be possible to prove existence, uniqueness and regularity for \eqref{eq:kryintro2} in the $L^p((0,T)\times\O;H^{s,q}(\T))$-setting with $q\in (1, \infty)$ under the same conditions on $p$ and $\beta$, but this is more technical. Details in the
$L^q$-case are presented for another equation in Section \ref{sec:LpLq}. Note
that with similar arguments one can also consider \eqref{eq:kryintro2} on
$\R$.
\end{example}

\begin{remark}\
The argument in Step 1 of the proof of Theorem \ref{thm:wellposed} also makes sense if the number $p$ satisfies $0<p\leq 1$. However, one needs further study to see whether $b V$  or $B U$ are stochastically integrable in this case. The definitions of $D_a(1-\frac1p,p)$ and $D_A(1-\frac1p,p)$ could be extended by just allowing $p\in(0,1)$ in \eqref{eq:realintDA}. It is interesting to see that if $p\downarrow 0$, the condition \eqref{eq:condsigma} becomes
    $2\alpha^2-2\beta^2 <1$.
\end{remark}

\section{Sharpness of the condition in the $L^p(L^2)$-setting\label{sec:LpL2}}

Below we consider the case when  the operator $A$ from  Theorem \ref{thm:wellposed} and \eqref{eq:SPDEFA} is the periodic Laplacian, i.e. the Laplacian with periodic boundary conditions. We will show below that in this case condition \eqref{eq:condsigma} is optimal.
Consider the following SPDE on the torus $\T = [0,2\pi]$.
\begin{equation}\label{eq:SPDE}
\left\{
  \begin{array}{rlll}
    du(t) &= \Delta u(t,x) \, dt &  + \ 2 \alpha D u(t,x) \, dW(t) \\ & &  + \ 2 \beta |D| u(t,x) \, dW(t), & t\in \R_+, x\in \T, \\
    D^k u(t,0) &= D^k u(t,2\pi), & & t\in \R_+, \  k\in \{0,1\} \\
    u(0,x)&=u_0(x), &  & x\in \T.
  \end{array}
\right.
\end{equation}
Here $D$ denotes the derivative with respect to $x$, $|D| = (-\Delta)^{1/2}$, the initial value $u_0:\O\to \mathcal{D}'(\T)$ is $\F_0$-measurable and $\alpha,\beta\in \R$ are constants not both equal to zero.

Let $X=H^{s,2}(\T)$ and $s\in \R$.
Then  problem \eqref{eq:SPDE} in the functional analytic formulation becomes
\begin{equation}\label{eq:SPDEFAappl}
\left\{
  \begin{array}{rll}
    dU(t) + A U(t) \, dt &=  2B U(t) \, dW(t) , & t\in \R_+,  \\
    U(0)&=u_0.
  \end{array}
\right.
\end{equation}
Here $A = -\Delta$ with domain $D(A) = H^{s+2,2}(\T)$ and $B:H^{s+2,2}(\T)\to
H^{s+1,2}(\T)$ is given by $B = \alpha D  + \beta |D|$ with $D(B) = D(D) =
H^{s+1,2}(\T)$. The connection between $u$ and $U$ is given by $u(t,\omega,x)
= U(t,\omega)(x)$. A process $u$ is called {\em an
$L^p(H^{s,2})$-solution to \eqref{eq:SPDE} on $[0,\tau)$} if $U$ is an
$L^p(H^{s,2})$-solution of \eqref{eq:SPDEFAappl} on $[0,\tau)$.

\begin{theorem}\label{thm:sharpex}
Let $p\in(1,\infty)$ and let $s\in\R$.
\begin{enumerate}[(i)]
\item\label{a:i} If $2\alpha^2+2\beta^2(p-1)<1$, then for every $u_0\in L^p(\O,\F_0;B^{s+2-\frac{2}{p}}_{2,p}(\T))$ there exists a unique $L^p(H^{s,2})$-solution $U$ of \eqref{eq:SPDEFAappl} on $[0,\infty)$. Moreover, for every $T<\infty$ there is a constant $C_T$ independent of $u_0$ such that
\begin{align}
\label{eq:Lptime-ex} \|U\|_{L^p((0,T)\times\O;H^{s+2,2}(\T))} &\leq C_T \|u_0\|_{L^p(\O;B^{s+2-\frac{2}{p}}_{2,p}(\T))}.
\end{align}
If, additionally, $p\in [2, \infty)$, then for every $T<\infty$ there is a constant $C_T$ independent of $u_0$ such that
\begin{equation}
\label{eq:Ctime-ex} \|U\|_{L^p(\O;C([0,T];B^{s+2-\frac{2}{p},2}_{2,p}(\T)))} \leq C_T \|u_0\|_{L^p(\O;B^{s+2-\frac{2}{p}}_{2,p}(\T))}.
\end{equation}
\item\label{a:ii} If $2\alpha^2+ 2\beta^2(p-1)>1$, and
\begin{equation}\label{eq:u0emacht}
u_0(x) = \sum_{n\in \Z\setminus\{0\}} e^{-n^{2}} e^{inx}, \  \  x\in \T,
\end{equation}
then there exists a unique $L^p(H^{s,2})$-solution of \eqref{eq:SPDEFAappl} on $[0,\tau)$, where $\tau = \big(2\alpha^2+2\beta^2(p-1)-1\big)^{-1}$. Moreover, $u_0\in \bigcap_{\gamma\in \R} B^{\gamma}_{2,p}(\T) = C^\infty(\T)$ and
\begin{align}
\label{eq:negativeL2} \limsup_{t\uparrow \tau} \|U(t)\|_{L^p(\O;H^{s,2}(\T))} & = \infty.
\end{align}
If, additionally, $p\in [2, \infty)$, then also
\begin{equation}
\label{eq:firstnegativeL2} \|U\|_{L^p((0,\tau)\times\O;H^{s+2,2}(\T))}  = \infty
\end{equation}

\end{enumerate}
\end{theorem}

\begin{remark}
Setting
\[u_0(x) = \sum_{n\in \Z\setminus\{0\}} e^{-\delta n^{2}} e^{inx},\]
where $\delta>0$ is a parameter, one can check that the assertion in \eqref{a:ii} holds if one takes
\[\tau = \delta/(2\alpha^2+2(p-1)\beta^2-1).\]
This shows how the {\em nonrandom explosion time} varies for some class of initial conditions.
\end{remark}

\begin{proof}
\eqref{a:i}: This follows from Theorems \ref{thm:wellposed} and \ref{thm:wellposed2}, \eqref{eq:interp} and the text below \eqref{eq:realintDA}.

\eqref{a:ii}: Taking the Fourier transforms on $\T$ in \eqref{eq:SPDEFAappl} one obtains the following family of scalar-valued SDEs with $n\in \Z$:
\begin{equation}\label{eq:SPDEFourier2}
\left\{
  \begin{array}{rll}
    d v_n(t) &= -n^{2} v_n (t) \, dt + (2 i\alpha n + 2\beta |n|) v_n(t) \, dW(t), & t\in \R_+, \\
    v_n(0)&=a_n,
  \end{array}
\right.
\end{equation}
where $v_n(t) = \F(U(t))(n)$ and $a_n = e^{-n^2}$.
Fix $n\in \Z$. It is well-known from the theory of SDEs that \eqref{eq:SPDEFourier2} has a unique solution $v_n:\R_+\times \O\to \R$ given by
\begin{equation}\label{eq:defvn}
v_n(t) = e^{- t(n^2 + 2b^2_n)} e^{2\beta |n| W(t)} e^{2\alpha i n W(t)} a_n,
\end{equation}
where $b_n = \beta |n|+i \alpha n$. Now let $U:\R_+\times\O\to L^0(\T)$ be
defined by
\begin{align}\label{def:Usol2}
U(t,\omega)(x) = \sum_{n\in \Z} v_n(t,\omega) e^{i n x}.
\end{align}
Clearly, if an $L^p(H^{s,2})$-solution exists, it has to be of
the form \eqref{def:Usol2}. Hence uniqueness is obvious.

Let $T<\tau$ and let $t\in [0,T]$. As in \eqref{eq:isometryU} in the proof of Theorem \ref{thm:wellposed} (with $\eps=0$), one has
\begin{equation}\label{def:UH2}
\begin{aligned}
\|U(t,\omega)\|_{H^{s+2,2}(\T)}^2
= e^{2 h(t,\omega)} \sum_{n\in \Z\setminus\{0\}}(n^2+1)^{s+2} e^{-f(t) (|n| - g(t,\omega))^2} e^{-2n^{2}}
\\
= 2 e^{2 \tilde{h}(t,\omega)}\sum_{n\geq 1} (n^2+1)^{s+2} e^{-\tilde{f}(t) (n - \tilde{g}(t,\omega))^2},
\end{aligned}
\end{equation}
where in the last step we used the symmetry in $n$ and where for the term $\widehat{u}_0(n)= e^{-2n^{2}}$ we have introduced the following functions $\tilde{f}$, $\tilde{g}$ and $\tilde{h}$:
\[\tilde{f}(t) = 2 (t+2\theta t + 1) ,  \ \ \tilde{g}(t,\omega) = \frac{\beta W(t,\omega)}{t+2\theta t + 1},  \ \ \tilde{h}(t,\omega) = \frac{\beta^2 |W(t,\omega)|^2}{t+2\theta t +1},\]
where $\theta = \beta^2-\alpha^2$.
Note that for $t<\tau$  we have $t+2\theta t + 1\geq \gamma$, where
\[\gamma= \left\{
            \begin{array}{ll}
              1 & \hbox{if $2\alpha^2 - 2\beta^2\leq 1$,} \\
              T+2\theta T + 1 & \hbox{if $2\alpha^2 - 2\beta^2>1$.}
            \end{array}
          \right.
\]
The proof will be split in two parts. We prove the existence and regularity in
\eqref{a:ii} for all $s\geq -2$ and $t<\tau$. The blow-up of \eqref{a:ii}
will be proved for all $s<-2$. Since $H^{s,2}(\T)\hookrightarrow
H^{r,2}(\T)$ if $s>r$, this is sufficient.

Assume first that $s\geq -2$. Let $W(t,\omega)\geq 0$ and let $m\in \N$ be the unique integer such that
$m-1<\tilde{g}(t,\omega)\leq m$. Then one has
\begin{align*}
\sum_{n\geq 1} (n^2+1)^{s+2} e^{-\tilde{f}(t) (n -
\tilde{g}(t,\omega))^2} & \leq \sum_{n\geq 1} (n^2+1)^{s+2} e^{-\gamma (n - m)^2}
\\ & = \sum_{k\geq -m+1} ((k+m)^2+1)^{s+2} e^{-\gamma k^2}
\\    & \leq \sum_{k\geq -m+1} ((k+\tilde{g}(t,\omega)+1)^2+1)^{s+2} e^{-\gamma k^2}
\\    & \leq \sum_{k\in \Z} ((|k|+\tilde{g}(t,\omega)+1)^2+1)^{s+2} e^{-\gamma k^2}
\\ & \leq C_{s} \tilde{g}(t,\omega)^{2s+4} \sum_{k\in \Z} e^{-\gamma k^2} + C_{s}\sum_{k\in \Z} (|k|+1)^{2s+4} e^{-\gamma k^2}
\\ &  = C'_{s,\gamma}(\tilde{g}(t,\omega)^{2s+4} +1).
\end{align*}
Similarly, if $W(t,\omega)< 0$ one has
\[\sum_{n\geq 1} (n^2+1)^{s+2} e^{-\tilde{f}(t) (n - \tilde{g}(t,\omega))^2} \leq \sum_{n\geq 1} (n^2+1)^{s+2} e^{-\gamma k^2 }
\]

Hence by \eqref{def:UH2} and the previous estimate we infer that
\begin{align}\label{eq:estUgh}
\E \|U(t)\|_{H^{s+2,2}(\T)}^{p} \leq
C'_{s,\gamma}\E\big[ (|\tilde{g}(t)|^{2s+4} +1) e^{p\tilde{h}(t)}\big].
\end{align}
By the definition of the function $\tilde{h}$, the the RHS of \eqref{eq:estUgh} is finite if and only if  $\frac{p \beta^2 t}{t+2\theta
t+1}<\frac12$. This is equivalent with $\big[2(p-1)\beta^2+ 2 \alpha^2
-1\big]t<1$. Since by assumption $2(p-1)\beta^2+2\alpha^2-1>0$, the latter is
satisfied, because $t< \tau$.

Finally, we claim that $U\in L^p((0,T)\times\O;H^{s+2,2}(\T))$. Indeed, for all $0<t\leq T$ one has
\begin{align*}
\E\big[ (|\tilde{g}(t)^{2s+4}| +1) e^{p\tilde{h}(t)} \big]& = \E \Big[\Big(\Big[\frac{\beta |W(1)|}{\sqrt{t}+2\theta \sqrt{t} + t^{-1/2}}\Big]^{2s+4} +1\Big) e^{p\frac{\beta^2 |W(1)|^2}{1+2\theta +t^{-1}}} \Big]
\\ & \leq \E \Big[\Big(\Big[\frac{\beta |W(1)|}{\gamma t^{-1/2} }\Big]^{2s+4} +1\Big) e^{p\frac{\beta^2 |W(1)|^2}{1+2\theta +T^{-1}}}\Big]
\\ & \leq \E \Big[\Big(T^{s+2} (\beta/\gamma)^{2s+4} |W(1)|^{2s+4} +1\Big) e^{p\frac{\beta^2 |W(1)|^2}{1+2\theta +T^{-1}}}\Big] =(*)
\end{align*}
Since $(*)$ is independent of $t$ and finite by the assumption on $T$, the claim follows.
Now the fact that $U$ is a strong solution on $[0,T]$ can be checked as in Theorems \ref{thm:wellposed} and \ref{thm:wellposed2}.

We will show that for all $s<-2$ one has
\begin{equation}\label{eq:supinfinite}
\limsup_{t\uparrow \tau} \E\|U(t)\|_{H^{s+2,2}(\T)}^p = \infty.
\end{equation}
As observed earlier the blow-up in \eqref{eq:negativeL2}  follows from the above. Indeed, this is clear from the fact that the space $H^{\delta,2}(\T)$ becomes smaller as $\delta$ increases.
To prove \eqref{eq:supinfinite}, fix $t\in [0,\tau)$ and assume
$W(t,\omega)>0$. Let $m\geq 1$ be the unique integer such that
$m-1<\tilde{g}(t,\omega)\leq m$. Then one has
\begin{align*}
\|U(t,\omega)\|_{H^{s+2,2}(\T)}^2 &= 2 e^{2 \tilde{h}(t,\omega)}\sum_{n\geq 1} (n^2+1)^{s+2} e^{-\tilde{f}(t) (n - \tilde{g}(t,\omega))^2}
\\ & \geq 2 (m^2+1)^{s+2} e^{-\tilde{f}(t)} \geq ((\tilde{g}(t,\omega)+1)^2+1)^{s+2} e^{-\tilde{f}(t)}.
\end{align*}
Hence we obtain
\begin{align*}
\E &\|U(t)\|_{H^{s+2,2}(\T)}^{p} \\ & \geq  2e^{-\tilde{f}(t)p/2}
\int_{\{W(t)>0\}} ((\tilde{g}(t,\omega)+1)^2+1)^{\frac{s}{2} p + p}  e^{p\tilde{h}(t)} \, d\P
\\ & = 2e^{-\tilde{f}(t)}
\int_{\{W(1)>0\}} \Big(\frac{\beta W(1)}{\big(\sqrt{t}+2\theta \sqrt{t} +
t^{-1/2}} + 1\big)^2 + 1\Big)^{\frac{s p}{2}+ p} e^{\frac{p \beta^2 |W(1)|^2}{1+2\theta
+t^{-1}}} \, d\P
\\ & \geq 2 e^{-\tilde{f}(t)}\int_{\{W(1)>0\}} \Big(\Big(\frac{\beta
W(1)}{\gamma t^{-1/2}}+1\Big)^2 + 1\Big)^{\frac{s p}{2} + p} e^{ p\frac{\beta^2
|W(1)|^2}{1+2\theta+t^{-1}}} \, d\P
\\ & \geq 2 e^{-\tilde{f}(t)}\int_{\{W(1)>0\}} \Big(\Big(\frac{\beta
W(1)}{\gamma T^{-1/2}}+1\Big)^2 + 1\Big)^{\frac{s p}{2} + p} e^{ p\frac{\beta^2
|W(1)|^2}{1+2\theta+t^{-1}}} \, d\P
\end{align*}
The latter integral is infinite if $t=\tau$. Now
\eqref{eq:supinfinite} follows from the monotone convergence theorem and the
last lower estimate for $\E\|U(t)\|_{H^{s+2,2}(\T)}^p$.

Finally, we prove \eqref{eq:firstnegativeL2} for $p\in [2, \infty)$. Note that if $U\in L^p((0,\tau)\times\O;H^{r+2, 2}(\T))$ for some $r>s +\frac2p$,
then by using the mild formulation as in Step 2 of the proof of Theorem \ref{thm:wellposed} one obtains that
\[U\in L^p(\O; C([0,\tau];B^{r+2-\frac2p}_{2,p}(\T)))\hookrightarrow L^p(\O; C([0,\tau];H^{s+2, 2}(\T))),\]
where the embedding follows from Section \ref{sec:fspaces}. This would contradict
\eqref{eq:supinfinite}.
\end{proof}

\begin{remark}
From the above proof one also sees that if $2\alpha^2- 2\beta^2>1$, then
$\|U(t,\omega)\|_{H^{s,2}(\T)} = \infty$ for $t>(2\alpha^2- 2\beta^2-1)^{-1}$.
Indeed, this easily follows from \eqref{def:UH2} and the fact that
$\tilde{f}(t)<0$. Apparently, for such $t$, parabolicity is violated. On the other hand, if $2\alpha^2- 2\beta^2<1$, but $2\alpha^2+ 2\beta^2(p-1)<1$, the above proof shows that the ill-posedness is due to lack of $L^p(\O)$-integrability.
\end{remark}

\begin{remark}\label{rem:sigmaphiex}
The above theorem has an interesting consequence. Let
$2\alpha^2<1$ and let $\beta$ be arbitrary. If $p\in (1, \infty)$ is so
small that $2\alpha^2+2\beta^2(p-1)<1$, then \eqref{eq:SPDE} is well-posed.
\end{remark}

\section{Well-posedness and sharpness in the $L^p(L^q)$-setting\label{sec:LpLq}}

In this section we show that the problem \eqref{eq:SPDEFAappl} can also be
considered in an $L^q(\T)$-setting. The results are quite similar, but the
proofs are more involved, due to lack of orthogonality in $L^q(\T)$. Instead of using
orthogonality, we will rely on the Marcinkiewicz multiplier theorem, see Theorem  \ref{thm:marcin}.

Let $q\in (1, \infty)$ and $s\in \R$ and let $X = H^{s,q}(\T)$. Using Proposition \ref{prop:NVW} and Remark \ref{rem:isomUMD} one can extend Definitions
\ref{def:strong}, \ref{def:mild}, \ref{def:Lp} and Proposition
\ref{prop:equiv}. Here instead of $B(U) \in L^0(\O;L^2(0,T;X))$ (with $X = H^{s,2}(\T)$) in Definitions \ref{def:strong} and \ref{def:mild} (ii) one should assume $B(U)\in L^0(\O;H^{s,q}(\T;L^2(0,T)))$. In that way the stochastic integrability is defined as below Proposition \ref{prop:NVW}. This will used in the next theorem.

Concerning $L^p(H^{s,q})$-solutions one has the following.
\begin{theorem}\label{thm:Lq}
Let $p, q\in (1, \infty)$ and $s\in \R$ be arbitrary.
\begin{enumerate}[(i)]
\item\label{a:iLpLq} If $2\alpha^2+2\beta^2(p-1) <1$, then for every $u_0\in L^p(\O,\F_0;B^{s+2-\frac{2}{p}}_{q,p}(\T))$
there exists a unique $L^p(H^{s,q})$-solution $U$ of \eqref{eq:SPDEFAappl} on
$[0,\infty)$. Moreover, for every $T<\infty$ there is a constant $C_T$
independent of $u_0$ such that
\begin{align}
\label{eq:Lptime-ex2}
\|U\|_{L^p((0,T)\times\O;H^{s+2,q}(\T))} &\leq C_T
\|u_0\|_{L^p(\O;B^{s+2-\frac{2}{p}}_{q,p}(\T))}.
\end{align}
If, additionally, $q\geq 2$ and $p>2$, or $p=q=2$, then for every $T<\infty$ there is a constant $C_T$
independent of $u_0$ such that
\begin{equation}
\label{eq:Ctime-ex2}
\|U\|_{L^p(\O;C([0,T];B^{s+2-\frac{2}{p},q}_{q,p}(\T)))} \leq C_T \|u_0\|_{L^p(\O;B^{s+2-\frac{2}{p}}_{2,p}(\T))}.
\end{equation}

\item\label{a:iiLpLq} If $2\alpha^2 + 2\beta^2(p-1)>1$, and
\begin{equation}\label{eq:u0LpLq}
u_0(x) = \sum_{n\in \Z\setminus\{0\}} e^{-n^2} e^{-inx},
\end{equation}
then there exists a unique $L^p(H^{s,q})$-solution of \eqref{eq:SPDEFAappl} on $[0,\tau)$, where $\tau = (2\alpha^2 + 2(p-1)\beta^2-1)^{-1}$.
Moreover, $u_0\in \bigcap_{\delta\in \R} B^{\delta}_{q,p}(\T) = C^\infty(\T)$ and one has
\begin{align}
\label{eq:negativeLq} \limsup_{t\uparrow \tau} \|U(t)\|_{L^p(\O;H^{s,q}(\T))} & = \infty.
\end{align}
If, additionally, $q\geq 2$ and $p>2$, or $p=q=2$, then also
\[\|U\|_{L^p((0,\tau)\times\O;H^{s+2,q}(\T))}  = \infty.\]
\end{enumerate}
\end{theorem}

\begin{proof}
\

\eqref{a:iiLpLq}: Since $\T$ is a bounded domain this is a consequence of Theorem
\ref{thm:sharpex} \eqref{a:ii} (with different choices of $s$), and the embeddings $H^{s,q}(\T)\hookrightarrow
H^{s,2}(\T)$ and, see \cite[Theorems 3.5.4 and 3.5.5]{SchmTr},
\[H^{s,2}(\T) \hookrightarrow H^{s-\frac12+\frac1q,q}(\T).\]

\eqref{a:iLpLq}:  The solution $U$ is again of the form \eqref{def:Usol2}. To
prove the estimates in \eqref{a:iLpLq} we apply Theorem \ref{thm:marcin}. Let
$\eps\in (0,\tfrac12)$ be such that $2\alpha^2 + 2\beta^2(p-1)<1-2\eps$. Let $r
= 1-2\eps$. With similar notation as in the proof of Theorem \ref{thm:sharpex}
and with $a_n = \F(u_0)(n)$, let
\begin{align}\label{eq:vnLq}
v_n(t) = e^{- t(n^2+ 2b^2_n)} e^{2\beta |n| W(t)} e^{2\alpha i n W(t)} a_n
 =e^{h(t,\omega)} m_n(t,\omega) e^{-\frac{\eps n^2}{2}t} \widetilde{a}_n(t),
\end{align}
where $b_n = \beta |n| + i\alpha n$, $\widetilde{a}_n(t) = e^{ 2\alpha i n W(t)}a_n$ and $m = m^{(1)} m^{(2)}$ with
\[m_n^{(1)}(t,\omega) =  e^{-\frac12 f(t) [|n|-g(t,\omega)]^2},  \ \ \ m_n^{(2)}(t,\omega) =  e^{k_n(t,\omega)}.\]
Here $f, g, h$ are given by
\[f(t) = 2 (r+2\theta)t,  \ \ g(t,\omega) = \frac{\beta W(t,\omega)}{(r+2\theta)t},  \ \ h(t,\omega) = \frac{\beta^2 |W(t,\omega)|^2}{(r+2\theta)t},\]
where $\theta = \beta^2-\alpha^2$ and
\[k_n(t,\omega) = -\frac{\eps n^2}{2}t + 4i \beta\alpha t n^2.\]
By Facts \ref{fact:mult} (ii) one has
\begin{equation}\label{eq:prodm1m2}
\|m(t,\omega)\|_{\M}\leq \|m^{(1)}(t,\omega)\|_{\M} \|m^{(2)}(t,\omega)\|_{\M}.
\end{equation}
Let $A= \{0, 1, 2, \ldots\}$ and $B = \Z\setminus A$. For the first term we have
\begin{align*}
\|m^{(1)}(t,\omega)\|_{\M} & \leq
\|\one_{A}\, m^{(1)}(t,\omega)\|_{\M} +
\|\one_{B}\, m^{(1)}(t,\omega)\|_{\M}
\\ & = \|\one_{A}\, m^{(1)}(t,\omega)\|_{\M} +
\|\one_{A\setminus\{0\}}\, m^{(1)}(t,\omega)\|_{\M}
\\ & \leq 2\|\one_A \, m^{(1)}(t,\omega)\|_{\M}
+ \|\one_{\{0\}}\, m^{(1)}(t,\omega)\|_{\M}
\\ & \leq 3\|\one_{A} \, m^{(1)}(t,\omega)\|_{\M},
\end{align*}
where we used Facts \ref{fact:mult} (iv) in the last step.

Let $r(t,\omega) = g(t,\omega) - g_0(t,\omega)$ with $g_0(t,\omega) = \lfloor g(t,\omega)\rfloor$, and let $m^{(3)}_n(t,\omega) =e^{-\frac12 f(t) (n-r(t,\omega))^2}$.
Let $A_{g(t,\omega)} = \{n\in \Z: n\geq -g_0(t,\omega)\}$. By Facts \ref{fact:mult} (i) and (iii) one sees that
\begin{align*}
\|\one_{A} \, m^{(1)}(t,\omega)\|_{\M} & = \|\one_{A_{g(t,\omega)}} \, m^{(3)}(t,\omega)\|_{\M}
\\ & \leq \|\one_{A_{g(t,\omega)}}\|_{\M} \|m^{(3)}(t,\omega)\|_{\M}
\\ & = \|\one_{A}\|_{\M} \|m^{(3)}(t,\omega)\|_{\M}
\\ &\leq C_{1,q} C_{2,q},
\end{align*}
where we used Theorem \ref{thm:marcin} and \eqref{eq:Kmarcin} applied to $m^{(3)}$ (note that $r(t,\omega)\in [0,1]$).

Therefore, we find that with $K_q = 3 C_{1,q} C_{2,q}$ one has
\begin{equation}\label{eq:m1est}
\|m^{(1)}(t,\omega)\|_{\M}\leq K_q.
\end{equation}

To estimate $\|m^{(2)}(t,\omega)\|_{\M}$, let $\zeta(\xi) = e^{-\frac{\eps \xi^2}{2}t + i 4 \beta\alpha t \xi^2}$, $\xi\in \R$. To check that \eqref{eq:Kmarcin} is finite, first note that $\zeta$ is uniformly bounded. Moreover, one has
\begin{align*}
|\zeta'(\xi)| & = C \eps |\xi| t e^{-\frac{\eps \xi^2}{2}t}, \ \ \xi\in \R\setminus\{0\},
\end{align*}
where $C=\Big(1+ \frac{64\beta^2\alpha^2}{\varepsilon^2}\Big)^{1/2}$.
To estimate $\int_{2^{n-1}}^{2^n} |\zeta'(\xi)| \, d\xi$ and $\int_{-2^{n}}^{-2^{n-1}} |\zeta'(\xi)| \, d\xi$, by symmetry it suffices to consider the first one. We obtain
\begin{align*}
\int_{2^{n-1}}^{2^n} |\zeta'(\xi)| \, d\xi & \leq C \int_{2^{n-1}}^{2^n}  \eps \xi t  e^{-\frac{\eps \xi^2}{2}t}  \, d\xi
 = C  \big[ e^{- \eps 2^{2n-3} t} - e^{- \eps 2^{2n-1} t}\big]\leq C.
\end{align*}
Hence, Theorem \ref{thm:marcin} and \eqref{eq:Kmarcin} yield that $\|m^{(2)}(t,\omega)\|_{\M}\leq c_q C$. Therefore, from \eqref{eq:prodm1m2} and \eqref{eq:m1est} we can conclude that with $C_q = K_q c_q C$, one has
\begin{equation}\label{eq:m2est}
\|m(t,\omega)\|_{\M}\leq C_q.
\end{equation}

Let $c_n(t,\omega) = (1+n^2)^{(s+2)/2}  e^{-\frac{\eps n^2}{2}t}
\widetilde{a}_n(t,\omega)$, where we recall $\widetilde{a}_n(t) = e^{ 2\alpha i n W(t)}a_n$. Let $e_n(x) = e^{inx}$. Combining the definition of $U$,
\eqref{eq:vnLq} and \eqref{eq:m2est} we obtain that
\begin{align*}
\|U(t,\omega)\|_{H^{s+2,q}(\T)} &= \Big\|\sum_{n\in \Z} (1+n^2)^{(s+2)/2}
e^{h(t,\omega)} m_n(t,\omega) e^{-\frac{\eps n^2}{2}t} \widetilde{a}_n(t,\omega)
e_n \Big\|_{L^q(\T)}
\\ & =e^{h(t,\omega)} \Big\|\sum_{n\in \Z}  m_n(t,\omega) c_n(t,\omega) e_n\Big\|_{L^q(\T)}
\\ & \leq C_q e^{h(t,\omega)} \Big\|\sum_{n\in \Z}  c_n(t,\omega) e_n \Big\|_{L^q(\T)}
\\ & = C_q e^{h(t,\omega)} \Big\|\sum_{n\in \Z}  (1+n^2)^{(s+2)/2}  e^{-\frac{\eps n^2}{2}t}   a_n e_n(\cdot+ \alpha W(t,\omega)) \Big\|_{L^q(\T)}
\\ & = C_q e^{h(t,\omega)} \Big\|\sum_{n\in \Z}  (1+n^2)^{(s+2)/2}  e^{-\frac{\eps n^2}{2}t}  a_n e_n \Big\|_{L^q(\T)}
\\ & = C_q e^{h(t,\omega)} \Big\| (1-\Delta)^{(2+s)/2} e^{-\frac{t \eps}{2} \Delta} u_0\Big\|_{L^q(\T)}.
\end{align*}
By independence it follows that
\begin{align}\label{eq:pointwiseestU}
\E\|U(t)\|_{H^{s+2,q}(\T)}^q &\leq C_q^p \E (e^{p h(t,\cdot)}) \E \Big\| \big(1 -\Delta\big)^{(2+s)/2} e^{-\frac{t \eps}{2} \Delta} u_0\Big\|_{L^q(\T)}^p.
\end{align}
Recall that as before since $2\alpha^2  + 2\beta^2(p-1)<1$ one has $M^p:= \E (e^{p h(t,\omega)}) =\E (e^{p h(1,\omega)})<\infty$.
Integrating with respect to $t\in [0,T]$, yields that
\begin{align*}
\int_0^T \E\|U(t)\|_{H^{s+2,q}(\T)}^p \, dt& \leq C_q^p M^p \E \int_0^T \Big\| \big(1 -\Delta\big)^{(2+s)/2} e^{-\frac{t \eps}{2} \Delta} u_0\Big\|_{L^q(\T)}^p \, dt
\\ & \leq C C_q^p M^p  \E \|u_0\|_{B^{s+2-\frac{2}{p}}_{q,p}(\T)}^p,
\end{align*}
where the last estimate follows from \eqref{eq:interp} and
\eqref{eq:realintDA}. This proves \eqref{eq:Lptime-ex2}. The fact that $U$ is an $L^p$-solution of
\eqref{eq:SPDEFAappl} can be seen as in Theorems \ref{thm:wellposed} and \ref{thm:wellposed2}, but for convenience we present a detailed argument.

We check the conditions of Definitions \ref{def:strong} and \ref{def:Lp}. Recall that the second part of Definition \ref{def:strong} (ii) should be replaced by $B(U)\in L^0(\O;H^{s,q}(\T;L^2(0,T)))$ as explained at the beginning of Section \ref{sec:LpLq} (also see Remark \ref{rem:isomUMD}).

The strong measurability and adaptedness of $U:[0,T]\times\O\to H^{s,q}(\T)$ follows from the corresponding properties of $v_n$ defined in \eqref{eq:defvn} and the convergence of the series \eqref{def:Usol2} in $L^p(\O\times(0,T);H^{s,q}(\T))$ (which follows from \eqref{eq:Lptime-ex2}). Since, $D(A) = H^{s+2,q}(\T)$, the fact that $U\in L^p(\O\times(0,T);D(A))$ is immediate from \eqref{eq:Lptime-ex2}.

Next we show that $B(U)\in L^p(\O;H^{s,q}(\T;L^2(0,T)))$. By \eqref{eq:pointwiseestU} one has that for all $t\in (0,T]$,
\[(1-\Delta)^{(s+2)/2} U(t) = \sum_{n\in \Z} (1+n^2)^{(s+2)/2} v_n(t) e^{i n \cdot} \ \ \ \text{converges in $L^p(\O;L^q(\T))$}\]
\begin{equation}\label{eq:pointwisebound}
\begin{aligned}
\|U(t)\|_{L^p(\O;H^{s+2,q}(\T))}& \leq C \Big\| \big(1 -\Delta\big)^{(2+s)/2} e^{-\frac{t \eps}{2} \Delta} u_0\Big\|_{L^p(\O;L^q(\T))}
\\ & \leq C_{t} \|u_0\|_{L^p(\O;H^{s,q}(\T))}.
\end{aligned}
\end{equation}
Applying \eqref{eq:SPDEFourier2} yields that for each $t\in [0,T]$ and $n\in \Z$,
\begin{equation}\label{eq:solbeq2}
\begin{aligned}
\int_0^T (1+n^2)^{s/2} (2b_n v_{n}(s)\, d W(s) & = (1+n^2)^{s/2} v_{n}(T) - (1+n^2)^{s/2} v_n(0) \\ & \qquad+ \int_0^T (1+n^2)^{s/2} n^2 v_n(s)\, d s := \eta_n(t).
\end{aligned}
\end{equation}
Let for each $t\in [0,T]$, $\eta(t)\in L^p(\O;L^q(\T))$ be defined by $\eta(t)(x) = \sum_{n\in \Z} \eta_n(t) e^{in x}$. Then by \eqref{eq:pointwisebound} and \eqref{eq:solbeq2} for every $t\in [0,T]$,
\[\eta(t) = (1-\Delta)^{s/2} U(t) - (1-\Delta)^{s/2} u_0 + \int_0^t (1-\Delta)^{s/2} A U(s) \, ds,\]
is in $L^p(\O;L^q(\T))$.
Using \eqref{eq:pointwisebound} and \eqref{eq:Lptime-ex2} it follows that for all $t\in (0,T]$,
\begin{align*}
\|\eta(t)\|_{L^p(\O;L^q(\T))} & \leq \|U(t)\|_{L^p(\O;H^{s,q}(\T))}+ \|u_0\|_{L^p(\O;H^{s,q}(\T))} \\ & \qquad \qquad \qquad \qquad + \int_0^t \|A U(s)\|_{L^p(\O;H^{s,q}(\T))} \, d s
\\ & \leq C_{t}  \|u_0\|_{L^p(\O;H^{s,q}(\T))} +t^{1-\frac{1}{p}}  \|AU\|_{L^p((0,T)\times\O;H^{s,q}(\T))}
\\ & \leq C_{t,T} \|u_0\|_{L^p(\O;B^{s+2-\frac{2}{p}}_{2,p}(\T))} .
\end{align*}
We claim that $B(U)\in L^p(\O;H^{s,q}(\T;L^2(0,T)))$, and for all $t\in [0,T]$,
\[\int_0^t (1-\Delta)^{s/2} 2 B(U) \, d W(s) = \eta(t)\]
where the stochastic integral exists in $L^p(\O;L^q(\T))$, see Appendix \ref{sec:stochintLq}.
By \eqref{eq:solbeq2} and Lemma \ref{lem:suffstochint} (with $\phi = (1-\Delta)^{s/2}2B(U)$, $(\psi_n)_{n\in \Z} = ((1+n^2)^{s/2} 2b_n v_n)_{n\in \Z}$ and $\OO = \N$), the claim
follows, and from \eqref{eq:itoisom} we obtain
\begin{align*}
\| (1-\Delta)^{s/2} 2 B(U)\|_{L^p(\O;L^q(\T;L^2(0,T)))} & \leq c_{p,q} \|\eta_T\|_{L^p(\O;L^q(\T))} \\ & \leq c_{p,q} C_{T,T} \|u_0\|_{L^p(\O;B^{s+2-\frac{2}{p}}_{2,p}(\T))}.
\end{align*}

Finally, assume $q\geq 2$ and $p>2$ or $p=q=2$. To prove \eqref{eq:Ctime-ex2} one can proceed as in Step 2 of the proof of Theorem \ref{thm:wellposed}. Indeed, since $U$ is a mild solution as well, one has
\[U(t) = e^{tA} u_0 + \int_0^t e^{(t-s)A} B U(s) \, dW(s).\]
Now \eqref{eq:Ctime-ex2}  follows from \cite[Theorem 1.2]{NVW10}. Note that the assumptions of that theorem can be checked with Theorem \ref{thm:marcin} (cf.\ \cite[Example 10.2b]{KuWe}).
\end{proof}

\begin{remark}
Note that the proof that $B(U)$ is stochastically integrable can be simplified if $p,q\geq 2$. Indeed, the fact that $U$ is an $L^p$-solution, already implies that $B(U)\in L^p(\O;L^2(0,T;H^{s,q}(\T))$ and therefore, stochastic integrability can be deduced from Corollary \ref{cor:qgroter2}.
\end{remark}

\begin{appendix}
\section{Stochastic integrals in $L^q$-spaces\label{sec:stochintLq}}

Recall that if $X$ is a Hilbert space and $\phi:[0,T]\times\Omega \to
X$ is an adapted and strongly measurable process with $\phi\in
L^0(\O;L^2(0,T;X))$, then $\phi$ is stochastically integrable. Below we
explain stochastic integration theory of \cite{NVW1} in the cases $X = L^q$ with $q\in (1,
\infty)$ and also recall a weak sufficient condition for stochastic integrability.
The stochastic integration theory from \cite{NVW1} holds for the larger class of UMD Banach
spaces, but we only consider $L^q$-spaces below.
Even for the classical Hilbert space case $q=2$, the second equivalent
condition below is a useful characterization of stochastic
integrability.
\begin{proposition}\label{prop:NVW}
Let $(\OO, \Sigma, \mu)$ be a $\sigma$-finite measure space. Let $p,q\in (1, \infty)$. Let $T>0$.
For an adapted and strongly measurable process $\phi:[0,T]\times\Omega \to L^q(\OO)$ the following three assertions are equivalent.
\begin{enumerate}[$(1)$]
\item There exists a sequence of adapted step processes $(\phi_n)_{n\geq 1}$ such that
\begin{enumerate}[(i)]
\item $\displaystyle \limn \|\phi - \phi_n\|_{L^p(\O;L^q(\OO;L^2(0,T)))} = 0$,
\item $(\int_0^T \phi_n(t) \, d W(t))_{n\geq 1}$ is Cauchy sequence in $L^p(\O;L^q(\OO))$.
\end{enumerate}
\item There exists a random variable $\eta\in L^p(\O;L^q(\OO))$ such that for all sets $A\in \Sigma$ with finite measure one has
$(t,\omega) \mapsto \int_A \phi(t,\omega) \, d\mu\in L^p(\O;L^2(0,T))$, and
\[\int_A \eta \, d\mu  =  \int_0^T \int_A \phi(t) \, d\mu \, d W(t) \ \ \text{in $L^p(\O)$}.\]
\item $\|\phi\|_{L^p(\O;L^q(\OO;L^2(0,T)))}<\infty$.
\end{enumerate}
Moreover, in this situation one has $\displaystyle \limn \int_0^T \phi_n(t) \, d W(t) = \eta$, and
\begin{align}\label{eq:itoisom}
c_{p,q}^{-1} \|\phi\|_{L^p(\O;L^q(\OO;L^2(0,T)))} \leq \|\eta\|_{L^p(\O;L^q(\OO))} \leq C_{p,q} \|\phi\|_{L^p(\O;L^q(\OO;L^2(0,T)))}.
\end{align}
\end{proposition}

\begin{remark}
Note that the identity in (2) holds in $L^p(\O)$ by the Burkholder-Davis-Gundy inequalities. In order to check (3) one needs to take a version of $\phi$ which is scalar valued and depends on $[0,T]\times\O\times\OO$. Such a version can be obtained by strong measurability.
\end{remark}

A process $\phi$ which satisfies any of these equivalent conditions is called {\em
$L^p$-stochastically integrable on $[0,T]$}, and we will write
\[\int_0^T \phi(t) \, d W(t) = \eta.\]
It follows from (3) that $\phi$ is $L^p$-stochastically integrable on $[0,t]$ as well.
By the Doob maximal inequality, see \cite[Proposition 7.16]{Kal}, one additionally gets
\[c_{p,q}^{-1} \|\phi\|_{L^p(\O;L^q(\OO;L^2(0,T)))}  \leq \Big\|t\mapsto \int_0^t \phi(s) \, dW(s)\Big\|_{F} \leq C_{p,q} \|\phi\|_{L^p(\O;L^q(\OO;L^2(0,T)))},\]
where $F = L^p(\O;C([0,T];L^q(\OO)))$. Moreover, in \cite[Theorem 5.9]{NVW1}
it has been shown that Proposition \ref{prop:NVW} can be localized and it is
enough to assume $\phi\in L^0(\O;L^q(\OO;L^2(0,T)))$ in order to have
stochastic integrability.

\begin{proof}
The result follows from \cite[Theorem 3.6 and Corollary 3.11]{NVW1} with $H=\R$. For this let us note that (2) implies that
for all $g\in L^q(\OO)$ which are finite linear combinations of $\one_{A}$ with $\mu(A)<\infty$, one has
\begin{equation}\label{eq:dualityLq}
\lb \eta, g\rb  =  \int_0^T \lb \phi(t), g\rb \, d W(t) \ \ \text{in $L^p(\O)$},
\end{equation}
where we use the notation $\lb \cdot, \cdot\rb$ for the duality of $L^q(\OO)$ and $L^{q'}(\OO)$. By a limiting argument one can see that for all $g\in L^{q'}(\OO)$ one has $\lb \phi, g\rb\in L^p(\O;L^2(0,T))$, and \eqref{eq:dualityLq} holds. This is the equivalent condition in \cite[Theorem 3.6]{NVW1}. Moreover, it is well-known that either (1) or (3) imply that for all $g\in L^{q'}(\OO)$ one has $\lb \phi, g\rb\in L^p(\O;L^2(0,T))$. See \cite[Corollary 3.11]{NVW1} and reference given there.
\end{proof}

If $q\in [2, \infty)$ there is an easy sufficient condition for $L^p$-stochastic integrability.
\begin{corollary}\label{cor:qgroter2}
Let $(\OO, \Sigma, \mu)$ be a $\sigma$-finite measure space. Let $p\in (1, \infty)$ and $q\in [2, \infty)$. Let $T>0$.
Let $\phi:[0,T]\times\Omega \to L^q(\OO)$ be an adapted and strongly measurable process.
If $\|\phi\|_{L^p(\O;L^2(0,T;L^q(\OO)))}<\infty$, then $\phi$ is $L^p$-stochastically integrable on $[0,T]$ and
\[\Big\|\int_0^T \phi(t) \, d W(t)\Big\|_{L^p(\O;L^q(\OO))}\leq C_{p,q} \|\phi\|_{L^p(\O;L^2(0,T;L^q(\OO)))}.\]
\end{corollary}
This result can be localized, and it is sufficient to have $\phi\in
L^0(\O;L^2(0,T;L^q(\OO)))$ in order to a stochastic integral. In Corollary
\ref{cor:qgroter2} one can replace $L^q(\OO)$ by any space $X$ which has
martingale type $2$, see \cite{Brz1,Dettunpub,Nh,Ondrej,Seidopt}.

\begin{proof}
By Minkowski's integral inequality, see \cite[Lemma 3.3.1]{KwWo}, one has
\[L^p(\O;L^2(0,T;L^q(\OO)))\hookrightarrow \|\phi\|_{L^p(\O;L^q(\OO;L^2(0,T)))}.\]
Therefore, the result follows from Proposition \ref{prop:NVW}.
\end{proof}

The following lemma is used in Sections \ref{sec:LpL2} and \ref{sec:LpLq}.
\begin{lemma}\label{lem:suffstochint}
Let $(\OO, \Sigma, \mu)$ be a $\sigma$-finite measure space. Let $p\in (1, \infty)$ and $q\in (1, \infty)$. Let $T>0$.
Let $\phi:[0,T]\times\Omega \to L^q(\OO)$ be an adapted and strongly measurable process.
Assume the following conditions:
\begin{enumerate}[$(1)$]
\item Assume that there exist a measurable function $\psi:[0,T]\times\O\times\OO\to \R$ such that $\phi(t,\omega)(x) = \psi(t,\omega, x)$ for almost all $t\in [0,T]$, $\omega\in \O$ and $x\in \OO$, and for all $x\in \OO$, $\psi(\cdot, x)$ is adapted.
\item For almost all $x\in \OO$, $\psi(\cdot,x)\in L^p(\O;L^2(0,T))$.
\item Assume that there is a $\eta\in L^p(\O;L^q(\OO))$ such that
\[\eta(\omega)(x) =  \Big(\int_0^T \psi(t,x) \, d W(t)\Big)(\omega) \ \ \text{for almost all $\omega\in \O$, and $x\in \OO$}.\]
\end{enumerate}
Then $\phi$ is $L^p$-stochastically integrable on $[0,T]$ and
\[\int_0^T \phi(t) \, d W(t) = \eta.\]
\end{lemma}
\begin{proof}
Note that the stochastic integral in (3) is well-defined. Indeed, by the
adaptedness of $\phi$ and (1), one has for almost all $x\in \OO$,
$\psi(\cdot,x)$ is adapted. Therefore, (2) shows that for almost all $x\in
\OO$, $\int_0^T \psi(t,x) \, d W(t)$ exists in $L^p(\O)$, and by Doob's
maximal inequality and the Bukrholder--Davis--Gundy inequality, see
\cite[Theorem 17.7]{Kal}, one has
\begin{equation}\label{eq:DoobBDGpsi}
c_p^{-1} \|\psi(\cdot, x)\|_{E} \leq  \Big\|\int_0^T \psi(t,x) \, d W(t)\Big\|_{L^p(\O)}\leq C_p \|\psi(\cdot, x)\|_{E},
\end{equation}
where $E = L^p(\O;L^2(0,T))$

First assume $p\leq q$. Fix $A\in \Sigma$ with finite measure. We claim that
$\phi\in L^p(\O;L^2(0,T;L^1(A)))$. Indeed, one has
\begin{align*}
\Big\| &\int_A |\phi| \, d\mu\Big\|_{E}  =  \Big\|\int_A |\psi(\cdot,x)| \, d\mu(x)\Big\|_{E}
\\ &  {\leq}  \int_A \|\psi(\cdot, x)\|_{E}\, d\mu(x) & \text{(Minkowski's inequality)}
\\  & {\leq} c_p \int_A \Big\|\int_0^T \psi(t,x) \, d W(t)\Big\|_{L^p(\O)} \, d\mu(x) & \text{(by \eqref{eq:DoobBDGpsi})}
\\ & {\leq} c_p \mu(A)^{\frac{1}{q'}}\Big(\int_A \Big\|\int_0^T \psi(t,x) \, d W(t)\Big\|_{L^p(\O)}^q \, d\mu(x)\Big)^{1/q} & \text{(H\"older's inequality)}
\\ & {\leq}
c_p \mu(A)^{1/q'} \|\eta\|_{L^p(\O;L^q(\OO))} & \text{(Minkowski's inequality)}
\end{align*}
This proofs the claim. In particular, one has $\int_A \phi \, d\mu\in E$.

Note that by the stochastic Fubini theorem one has
\[\int_A \eta(\omega) \, d\mu  =  \Big(\int_0^T \int_A  \psi(t,x) \, d\mu(x) \, d W(t)\Big)(\omega) =  \Big(\int_0^T \int_A \phi(t) \, d\mu \, d W(t)\Big)(\omega).\]
Hence, Proposition \ref{prop:NVW} (2) implies the result.

Now assume $p>q$. Then by the above case on obtains that $\phi$ is $L^q$-stochastically integrable on $[0,T]$. Moreover,
\begin{align*}
\Big\|\int_A \phi(t) \, d\mu\Big\|_{L^p(\O;L^2(0,T)} & \leq c_p \Big\|\int_0^T \int_A \phi(t) \, d\mu \, d W(t)\Big\|_{L^p(\O)}
\\ & = \Big\|\int_A \eta\, d\mu\Big\|_{L^p(\O)} \leq \mu(A)^{\frac1{q'}} \|\eta\|_{L^p(\O;L^q(\OO))}.
\end{align*}
Therefore, another application of Proposition \ref{prop:NVW} (2) shows that $\phi$ is actually $L^p$-stochastically integrable on $[0,T]$.
\end{proof}

\begin{remark}\label{rem:isomUMD}
Let us explain how the above result can also be applied to $H^{s,q}(\T)$ which
is isomorphic to a $L^q(\T)$. Let $J:H^{s,q}(\T)\to L^q(\T)$ be an
isomorphism. Then for a process $\phi:[0,T]\times\O\to H^{s,q}(\T)$ let
$\tilde{\phi} = J \phi$. The above results can be applied to $\tilde{\phi}$.
Conversely, if $\tilde{\eta} = \int_0^T \tilde{\phi}(t) \, d W(t)$, then we
define
\[\eta = J^{-1}\tilde{\eta}.\]
Moreover, $\|\phi\|_{L^p(\O;H^{s,q}(\T;L^2(0,T)))}<\infty$ is equivalent to
stochastic integrability of $\phi$. It is well-known, see
\cite[8.24]{Stein93}, that $J$ extends to a isomorphism from
$H^{s,q}(\T;L^2(0,T)))$ into $L^q(\T;L^2(0,T))$.

In a similar way, the results extend to arbitrary $X$ which are isomorphic to
a closed subspace of any $L^q(\OO)$.
\end{remark}
\end{appendix}

\def\polhk#1{\setbox0=\hbox{#1}{\ooalign{\hidewidth
  \lower1.5ex\hbox{`}\hidewidth\crcr\unhbox0}}} \def\cprime{$'$}
\providecommand{\bysame}{\leavevmode\hbox to3em{\hrulefill}\thinspace}

\end{document}